%% file: main.tex
\documentclass{amsart}

\title{Non-escape of mass for arithmetic quantum limits on hyperbolic $4$-manifolds}

\author{Alexandre de Faveri}
\address{Department of Mathematics, Stanford University, Stanford, CA, USA}
\email{\url{afaveri@stanford.edu}}

\author{Zvi Shem-Tov}
\address{Einstein Institute of Mathematics, The Hebrew University of Jerusalem,  Israel}
\email{\url{zvi.shem-tov@mail.huji.ac.il}}

\usepackage{amssymb, amsmath, amsthm}
\usepackage{hyperref, enumitem}
\usepackage{thmtools}
\usepackage[capitalize, nameinlink]{cleveref}
\usepackage[space]{cite}
\usepackage[centering]{geometry}
\usepackage{xcolor, graphicx}
\usepackage{bbm}
\usepackage{mathrsfs}
\usepackage{pgf}
\usepackage{caption}
\usepackage{subcaption}
\usepackage{comment}
\usepackage[utf8]{inputenc}\DeclareUnicodeCharacter{2212}{-}

\input{mydefs}

\renewcommand{\H}{\mathbb{H}}
\renewcommand{\P}{\mathcal{P}}

\newcommand{\eps}{\varepsilon}
\newcommand{\1}{\mathbbm{1}}

\newcommand{\M}{\mathcal{M}}

\DeclareFontFamily{U}{mathx}{\hyphenchar\font45}
\DeclareFontShape{U}{mathx}{m}{n}{
      <5> <6> <7> <8> <9> <10>
      <10.95> <12> <14.4> <17.28> <20.74> <24.88>
      mathx10
      }{}
\DeclareSymbolFont{mathx}{U}{mathx}{m}{n}
\DeclareFontSubstitution{U}{mathx}{m}{n}
\DeclareMathAccent{\widecheck}{0}{mathx}{"71}
\DeclareMathAccent{\wideparen}{0}{mathx}{"75}

\declaretheorem[name=Theorem]{theorem}
\declaretheorem[name=Lemma]{lemma}
\declaretheorem[name=Proposition]{proposition}
\declaretheorem[name=Corollary]{corollary}

\declaretheorem[name=Remark]{remark}

\declaretheorem[name=Convention]{convention}

\newlist{theoremlist}{enumerate}{1}
\setlist[theoremlist]{label=(\alph{theoremlisti}),
                  ref=\thetheorem \ (\alph{theoremlisti}),
                  noitemsep}
                  
\newlist{lemmalist}{enumerate}{1}
\setlist[lemmalist]{label=(\alph{lemmalisti}),
                  ref=\thelemma \ (\alph{lemmalisti}),
                  noitemsep}
                  
\newlist{propositionlist}{enumerate}{1}
\setlist[propositionlist]{label=(\alph{propositionlisti}),
                  ref=\theproposition \ (\alph{propositionlisti}),
                  noitemsep}

\newlist{corollarylist}{enumerate}{1}
\setlist[corollarylist]{label=(\alph{corollarylisti}),
                  ref=\thecorollary \ (\alph{corollarylisti}),
                  noitemsep}
                                                      
\Crefname{listthm}{Theorem}{Theorems}
\Crefname{listlem}{Lemma}{Lemmas}
\Crefname{listprop}{Proposition}{Propositions}
\Crefname{listcor}{Corollary}{Corollaries}

\addtotheorempostheadhook[theorem]{\crefalias{theoremlisti}{listthm}}
\addtotheorempostheadhook[lemma]{\crefalias{lemmalisti}{listlem}}
\addtotheorempostheadhook[proposition]{\crefalias{propositionlisti}{listprop}}
\addtotheorempostheadhook[corollary]{\crefalias{corollarylisti}{listcor}}

\Crefname{section}{Section}{Sections}
\Crefname{subsection}{Section}{Sections}

\begin{document}

\begin{abstract}
    We make progress on the quantum unique ergodicity (QUE) conjecture for Hecke--Maass forms on a congruence quotient of hyperbolic $4$-space, eliminating the possibility of ``escape of mass'' for these forms.   
\end{abstract}

\maketitle


\section{Introduction}

\subsection{Background}

Let $M$ be a compact Riemannian manifold of negative sectional curvature and consider a sequence $\phi_j\in L^2(M)$ of eigenfunctions of the Laplace--Beltrami operator $\Delta$ with unit norm and eigenvalues going to infinity. Let $\mu_j$ denote the probability measure on $M$ defined by $\mu_j=\abs{\phi_j}^2d\vol_M$, where $d\vol_M$ is the uniform Riemannian probability measure on $M$. The quantum unique ergodicity (QUE) conjecture of Rudnick and Sarnak \cite{RS94} asserts that every weak-* limit of the $\mu_j$ is equal to $d\vol_M$. See \cite{Sar11} for a survey of the subject.

In its original form, the conjecture is open for any manifold $M$ as above. However, significant progress has been made for arithmetic manifolds, that is for $M=\Gamma\bs S$ where $S$ is a symmetric space and $\Gamma$ is an arithmetic lattice of isometries of $S$. In this case it is natural to consider joint eigenfunctions of both $\Delta$ and the Hecke operators, which are discrete averaging operators coming from the arithmetic structure of $M$.    

In the seminal work of Lindenstrauss \cite{Lin06}, QUE was proved for such eigenfunctions on certain (arithmetic) compact quotients of the hyperbolic plane $\H^2$. For $\SL_2(\Z)\bs \H^2$, which is not compact, Lindenstrauss proved that any weak-* limit as above must be \emph{proportional} to the uniform probability measure. To complete the proof of QUE in that case it was necessary to rule out the possibility of ``escape of mass'', i.e.\ to show that the constant of proportionality is equal to $1$. This was achieved by Soundararajan \cite{Sou10}. 
His work was later generalized by Zaman \cite{Zam12} to quotients of $\H^3$, a case in which arithmetic QUE was recently established by Silberman and the second named author \cite{S-TS22}. In a subsequent paper \cite{S-TS24} they obtained certain homogeneity results for weak-* limits on quotients of $\H^4$. In this paper we prove non-escape of mass for such limits.  

\subsection{Statement of results}\label{subsection:intro_results}

As a model for the hyperbolic $4$-space we take the upper half-space
\begin{equation*}
    \mathbb{H}^4=\{z=(x_0, x_1, x_2, y)\in\R^4\mid y>0\}, 
\end{equation*}
with the metric $ds^2 = y^{-2}(dx_0^2+dx_1^2+dx_2^2+dy^2)$. Let $G$ denote the group of orientation-preserving isometries of $\H^4$. Then $G$ is generated by translations $t_\beta$ and the inversion $s$ given by
\begin{align*}
    t_\beta:z\mapsto z+\beta \qquad \qquad \text{ and } \qquad \qquad s:z\mapsto -\frac{\overline{z}}{\norm{z}^2},
\end{align*}
where $\beta\in V^3:=\{z\in \R^4\mid y=0\}$, $\overline{z} := (x_0,-x_1,-x_2,-y)$, and $\norm{\cdot}^2$ is the usual $\ell^2$-norm in $\RR^4$. 

Let $\Gamma$ denote the subgroup of $G$ generated by $s$ and $t_\beta$, where $\beta$ ranges over integral vectors (that is $x_0,x_1,x_2\in\ZZ$).
Then $\Gamma$ is a non-uniform lattice in $G$ and we can endow the quotient space 
$$
X:=\Gamma\bs \H^4
$$
with the probability measure $dx$ coming from the Riemannian measure on $\HH^4$. 
In fact, $G$ is naturally identified as the central quotient of $\SV_2(\R) \isom {\rm{Spin}}(1,4)$ (see \cref{subsection:SV_n} for a discussion of $\SV_n$). Under this identification, $\Gamma$ corresponds to $\SV_2(\Z)$. Let $\mathcal{H}$ denote the corresponding algebra of Hecke operators on $X$ at primes $p\ne2$ (see \cref{subsection:Hecke_operators} for the definition). We are ready to state our main result.

\begin{theorem}\label{thm:non-escape_of_mass}
    Let $\phi_j\in L^2(X)$ be a sequence of joint eigenfunctions of $\mathcal{H}$ and $\Delta$ with unit norm and Laplace eigenvalues $\lambda_j \to \infty$.
    Suppose the probability measures $\mu_j=|\phi_j(x)|^2dx$ converge to $\mu$ in the weak-* topology. Then $\mu$ is a probability measure. 
\end{theorem}

Combining \cref{thm:non-escape_of_mass} with \cite{S-TS24} gives the following.

\begin{theorem}\label{thm:homogeneity}
   Let $\mu$ be a limit measure as in \cref{thm:non-escape_of_mass}. Then $\mu$ is a countable convex combination of measures, each of which is either the Riemannian probability measure on $X$ or the Riemannian probability measure on a totally geodesic hyperbolic submanifold of codimension $1$.
\end{theorem}

We remark that for $\SL_2(\Z)\bs \HH^2$ and $\SL_2(\Z[i])\bs \HH^3$, arithmetic QUE (with a quantitative rate of convergence) follows from GRH, since the Watson--Ichino triple product formula \cite{Wat02, Ich08} reduces it to subconvexity for certain $L$-functions (see \cite{Mar14, BHMWW24} for details). As far as the authors are aware, no such formula is available for congruence quotients of $\HH^n$ when $n \geq 4$, since they are no longer naturally identified with an adelic quotient of $\SL_2$ over a number field. To illustrate this difference, we remark that there are violations of the Ramanujan conjecture on the space $X=\SV_2(\Z) \backslash \H^4$ considered in this paper \cite{Pit05}. We refer to \cite{CLPS91, EGM90} and references therein for more on the relevant spectral theory.

\subsection{Sketch of the argument}

The proof of \cref{thm:non-escape_of_mass} follows the basic strategy of \cite{Sou10}, which is based on arithmetic relations between the Fourier coefficients. However, in our context these relations are significantly more complicated. For instance, unlike in the cases of $\H^2$ or $\H^3$, they involve an unbounded number of terms. After recalling the argument of Soundararajan (which we present with slight differences from the original), we will describe in more detail the main difficulties which arise in our setup and how we overcome them. The discussion in this subsection will be imprecise for conceptual clarity. 

\subsubsection{Soundararajan's proof for \texorpdfstring{$\SL_2(\Z) \backslash \H^2$}{}}

In \cite{Sou10} non-escape of mass is deduced from a uniform inequality for the Fourier coefficients $a(n)$ of a Hecke--Maass form, of the shape
\begin{equation*}
    s(x/y) \ll \frac{s(x)}{\sqrt{y}}  \qquad\qquad \text{for } s(x) := \sum_{n\leq x} \abs{a(n)}^2.
\end{equation*}

Let $\lambda(m)$ denote the eigenvalue of the Hecke operator $T_m$, so that $\lambda(m)a(n) = \sum_{d\mid (m, n)} a(mn/d^2)$. Consider primes $p \asymp \sqrt{y}$. By the relation
\begin{equation}\label{eq:tp-tp2}
\lambda(p)^2 = \lambda(p^2)+1,
\end{equation} 
at least one of $\lambda(p)$ and $\lambda(p^2)$ is bounded away from zero. For simplicity assume $|\lambda(p)|^2 \asymp L \gg 1$ for all $p \asymp \sqrt{y}$. Soundararajan splits $s$ into two sums $s^{<K}$ and $s^{\geq K}$, depending on whether $n$ has $<K$ or $\geq K$ prime factors $p \asymp \sqrt{y}$, for some parameter $K$. Then
\begin{equation}\label{eq:Sound_amplification}
    L \sqrt{y} \cdot s^{<K}\left(z\right) \approx \sum_{\substack{n \leq z \\ \#\{p \asymp \sqrt{y}\ :\ p \mid n\} < K}} \abs{a(n)}^2 \Bigg( \sum_{\substack{p \asymp \sqrt{y} \\ p \nmid n}} \abs{\lambda(p)}^2 \Bigg) \leq K \cdot s^{<K+1}\left(z \sqrt{y}\right)
\end{equation}
since there are $\approx \sqrt{y}$ primes $p \asymp \sqrt{y}$ (we are ignoring log factors) and $K$ will be small
. Applying this for $z = x/y$ and again for $z = x/\sqrt{y}$ gives $s^{<K}\left(x/y\right) \ll s(x)/\sqrt{y}$ as long as $K \ll L y^{1/4}$. 

Each $n$ present in $s^{\geq K}(x/y)$ is a multiple of some product $d = p_1 \cdots p_K$ of $K$ primes $p_i \asymp \sqrt{y}$. Note that $d \approx (\sqrt{y})^{K}$, and there are $\approx \binom{\sqrt{y}}{K}$ such $d$. Ignoring lower order terms in the Hecke relation, each one contributes 
        \begin{equation}\label{eq:Sound_multiple_of_p}
            \sum_{m \leq \frac{x}{yd}} |a(dm)|^2 \approx |\lambda(d)|^2 \cdot s\left(\frac{x}{yd}\right) \approx L^{K} \cdot s\left(\frac{x}{yd}\right) \approx L^K \cdot s\left(\frac{x}{y (\sqrt{y})^K}\right).
        \end{equation} 
        Fixing $x$ and using induction for the shorter sum $s(x/y(\sqrt{y})^K)$, where the base case is trivial since $s(x/y) = 0$ for $y>x$, leads to
        \begin{equation*}
            s^{\geq K}\left(\frac{x}{y}\right) \ll \binom{\sqrt{y}}{K} L^{K} \cdot s\left(\frac{x}{y(\sqrt{y})^{K}}\right) \ll \left(\frac{10\sqrt{y} L}{K y^{1/4}}\right)^{K} \frac{s(x)}{\sqrt{y}}.
        \end{equation*} 
        This succeeds if $K \geq 20 L y^{1/4}$, which is consistent with the restriction $K \ll L y^{1/4}$.

    \subsubsection{Our proof for \texorpdfstring{$\SV_2(\Z) \backslash \H^4$}{}} 

    The Fourier coefficients $A(\beta)$ are now indexed by $\beta\in \Z^3$, and we follow the strategy of deducing non-escape of mass from an inequality of the form
    \begin{equation*}
        S(x/y) \ll \frac{S(x)}{y^{1/8}} , \qquad\qquad \text{for } S(x) := \sum_{|\beta|^2 \leq x} \abs{A(\beta)}^2.
    \end{equation*}
    
    Instead of \eqref{eq:tp-tp2} we have a relation involving three Hecke operators, as $\SV_2$ has rank $2$ over $\Q_{p\ge3}$. There are two algebraically independent Hecke operators, with eigenvalues $\lambda_1(p)$ and $\lambda_2(p)$, and we consider a third (natural) one with eigenvalue $\lambda_3(p)$ and obtain a relation of the form 
    $$
    \lambda_1(p)^2 - \lambda_2(p) - \lambda_3(p) \approx 1,
    $$ 
    which forces $\max_{\ell}|\lambda_\ell(p)| \gg 1$. As in \cite{Sou10}, we split $S$ into two sums $S^\sharp$ and $S^\flat$ corresponding respectively to $\beta$ divisible by few or many primes $p \asymp y^{1/8}$ (we must also control prime powers but ignore this point for the sketch). 
    
    We will now highlight two obstructions which arise in the proof. First, assume for simplicity that $|\lambda_1(p)|^2 \asymp 1$ for all $p \asymp y^{1/8}$ (the case when the $|\lambda_1(p)|^2$ localize around a larger value turns out to be easier, and the case when they are small will be addressed below). Identify $\beta \in \Z^3$ with $\beta = b_1 i + b_2j + b_3k$. The first Hecke operator acts by
    \begin{align}\label{eq:fake_Hecke_1}
        \lambda_1(p) A(\beta) \approx A(p\beta) + A\left(\frac{\beta}{p}\right) + \frac{1}{\sqrt{p}} \sum_{|\alpha|^2 = p} A\left( \frac{\alpha \beta \overline{\alpha}}{p} \right),
    \end{align}
    where $\alpha$ runs over integral quaternions. To deal with $S^\flat$ as in \eqref{eq:Sound_multiple_of_p}, it is crucial to obtain for instance
    \begin{equation}\label{eq:fake_p_multiple_sum}
        \sum_{|\beta|^2 \leq z} |A(p\beta)|^2 \ll |\lambda_1(p)|^2 \cdot S(z) \asymp S(z).
    \end{equation}
    Note that isolating $A(p\beta)$ in \eqref{eq:fake_Hecke_1}, we now have to contend with the sum over $\alpha$, of length $\asymp p$. A naive application of Cauchy--Schwarz would lead to terms of the form $|A(p\delta)|^2$ appearing with multiplicity $\gg p$ (from each choice $\beta = \overline{\alpha}\delta \alpha$). We remedy this by adding appropriate weights: setting
    \begin{equation*}
        I(\beta) = \{ \alpha: |\alpha|^2 = p \text{ and } v_p(\alpha \beta \overline{\alpha}) > v_p(\beta)\},
    \end{equation*}
    unique factorization of quaternions gives $|I(\beta)| \leq 16$, so
    \begin{align}\label{eq:fake_Cauchy}
        \left| \sum_{|\alpha|^2 = p} A\left( \frac{\alpha \beta \overline{\alpha}}{p} \right) \right|^2 \ll \sum_{\alpha \in I(\beta)} \left|A\left( \frac{\alpha \beta \overline{\alpha}}{p} \right)\right|^2 + p\sum_{\alpha \notin I(\beta)} \left|A\left( \frac{\alpha \beta \overline{\alpha}}{p} \right)\right|^2.
    \end{align}
    This dampens the problematic terms, so that denoting by $m_1(\delta)$ and $m_2(\delta)$ the multiplicities coming respectively from the first and second terms in the RHS of \eqref{eq:fake_Cauchy} leads to the desired bound
    \begin{align*}
        \frac{1}{p}\sum_{|\beta|^2 \leq z} \left| \sum_{|\alpha|^2 = p} A\left( \frac{\alpha \beta \overline{\alpha}}{p} \right)\right|^2  \ll\sum_{|\delta|^2 \leq z} \left(\frac{m_1(\delta)}{p} + m_2(\delta)\right) |A(\delta)|^2 \ll S(z),
    \end{align*}
    since we show the crucial estimate $m_2(\delta) \ll 1$ using the restriction $\alpha\not\in I(\beta)$. 

    Let us now mention another issue when we are forced to use the $\lambda_2(p)$. For that, assume $|\lambda_1(p)| \ll 1$ and $|\lambda_2(p)|^2 \asymp L\gg 1$ for all $p \asymp y^{1/8}$. The second Hecke operator acts roughly by
    \begin{align}\label{eq:fake_Hecke_2}
        \lambda_2(p) A(\beta) \approx \frac{1}{\sqrt{p}} \sum_{|\alpha|^2 = p} \left[ A\left( \alpha \beta \overline{\alpha} \right) + A\left( \frac{\alpha \beta \overline{\alpha}}{p^2} \right) \right].
        \end{align}
    The amplification scheme of \eqref{eq:Sound_amplification} leads to
    \begin{align}\label{eq:fake_amplification}
        L y^{1/8} \cdot S^{\sharp}\left(\frac{x}{y}\right) \ll \sum_{\substack{|\beta|^2 \leq \frac{x}{y} \\ \beta \in \M(\Vec{K})}} |A(\beta)|^2 \cdot \Bigg( \sum_{\substack{p \asymp y^{1/8} \\ p \nmid \beta}} |\lambda_2(p)|^2\Bigg) \approx \sum_{\substack{|\beta|^2 \leq \frac{x}{y} \\ \beta \in \M(\Vec{K})}} \sum_{\substack{p \asymp y^{1/8} \\ p \nmid \beta}} \abs{\sum_{|\alpha|^2 = p} A\left(\alpha\beta \overline{\alpha}\right)}^2
    \end{align}
    for a certain set $\M(\Vec{K})$ defining $S^\sharp$, where we applied \eqref{eq:fake_Hecke_2} and the terms $A(\alpha\beta\overline{\alpha}/p^2)$ were ignored since they can be directly treated by the argument in \eqref{eq:fake_Cauchy}.

    At this point, using Cauchy--Schwarz as in \cite{Sou10} (with any choice of weights) is a bad move: it leads to terms $|A(\delta)|^2$ with very large weighted multiplicity (on average $\approx y^{1/8}$) for many large $\delta$ ($|\delta|^2 \asymp x/y^{3/4}$) with little detectable structure ($v_p(\delta) = 0$ for all $p\asymp y^{1/8}$). Indeed, one considers the $\delta$ which are a conjugate by one $\alpha$ for each $p\asymp y^{1/8}$ in certain congruence classes. This issue would not arise if $\M(\Vec{K})$ were defined to control multiplicity of representations as a conjugate by each $\alpha$, but this would require the bound \eqref{eq:fake_p_multiple_sum} with $p\beta$ replaced by $\alpha\beta\overline{\alpha}$, which seems difficult since isolating $A(\alpha\beta\overline{\alpha})$ in \eqref{eq:fake_Hecke_1} or \eqref{eq:fake_Hecke_2} becomes wasteful due to the small coefficient $\frac{1}{\sqrt{p}}$. 
    
    Instead, we trade conjugation by $\alpha$ for powers of $p$ by observing from \eqref{eq:fake_Hecke_1} and \eqref{eq:fake_Hecke_2} that
    \begin{align}
        \lambda_2(p)A(\beta) &\approx \lambda_1(p)A(p\beta) - A(p^2\beta) - A(\beta), \label{eq:combinatorial_relation_1} \\
        \lambda_2(p)A(\beta) &\approx \lambda_2(p)A(p^2\beta) + \lambda_1(p) A(p^3 \beta) - A(p^4 \beta) - A(p^2 \beta). \label{eq:combinatorial_relation_2}
    \end{align}
    Then \eqref{eq:combinatorial_relation_1} leads to the required bound, provided that $L$ is sufficiently large so that the contribution of the terms $A(\beta)$ is subsumed by the LHS of \eqref{eq:fake_amplification}. Similarly, \eqref{eq:combinatorial_relation_2} leads to the required bound as long as $L \ll 1$, since then the factor of $\lambda_2(p)$ on the RHS does not nullify the required gain of $L$ in \eqref{eq:fake_amplification}. The case which requires using $\lambda_3(p)$ presents no further obstructions.

\subsection*{Acknowledgements}

We thank Elon Lindenstrauss, Lior Silberman, Kannan Soundararajan, and Liyang Yang for fruitful discussions and encouragement. This material is based upon work supported by the National Science Foundation under grant no.\ DMS-1926686. ZS was partially supported by the ERC grant HomDyn no.\ 833423.


\section{Preliminaries}

\subsection{The group \texorpdfstring{$\SV_n(\R)$}{} and its action on \texorpdfstring{$\HH^{n+2}$}{}}\label{subsection:SV_n}

In this section we recall how to use Clifford algebras to describe orientation-preserving isometries of hyperbolic spaces. We begin by reviewing the basic definitions and facts about Clifford algebras, which will be used throughout the paper. The material below is standard; proofs may be found in \cite{Ahl86} and \cite{EGM87}, which we mostly follow.  

\subsubsection{Clifford algebras}

The Clifford algebra $C_n$ is the (unital) associative algebra
over $\R$ generated by elements $i_1,\dots,i_n$ with the relations $i_g i_h= -i_h i_g$ for $g\ne h$, $i^2_h=-1$, and no others.
Note that $C_0$ can be identified with $\R$, $C_1$ with $\CC$, and $C_2$ with the Hamilton quaternions $\bH$. In the last case, $i,j$, and $k$ correspond respectively to $i_1,i_2$, and $i_1i_2$. Each element $a\in C_n$ has a unique representation of the form $a = \sum a_I I$, where $a_I\in\RR$ and the sum is over all products $I=i_{h_1}\cdots i_{h_r}$ with $1\leq h_1<\dots<h_r\leq n$. The empty product is included and identified with the real number $1$, and its coefficient is referred to as the {\it real part} ${\rm{Re}} (a)$. The {\it norm} of $a$ is $N(a) := \sum a_I^2 \in \R_{\geq 0}$. Clearly $C_n$ is a vector space over $\R$ of dimension $2^n$. 

\subsubsection{Clifford vectors}

The Clifford elements of the
form $x=x_0+x_1i_1+\dots+x_ni_n$ are called \emph{vectors}. They form an $(n+1)$-dimensional subspace $V^{n+1}$ which we shall identify with $\RR^{n+1}$. There are three involutions in $C_n$, similar to complex conjugation. The \emph{main involution} is the automorphism $a\mapsto a'$ obtained by replacing every $i_h$ by $-i_h$. The \emph{reverse involution} $a\mapsto a^*$ is obtained by reversing the order of the factors in each $I=i_{h_1}\cdots i_{h_p}$. It is an anti-isomorphism, so $(ab)^*=b^*a^*$. These involutions can be combined into the anti-isomorphism $a \mapsto \overline{a} := {a'}^* = {a^*}'$. If $x$ is a vector then $x=x^*$ and $x'=\overline{x}$, so all three involutions preserve $V^{n+1}$. If $x$ and $y$ are vectors one verifies that $x\overline{y}+y\overline{x}=2 \langle x,y \rangle$, where $\langle\cdot,\cdot\rangle$ is the usual inner product on $\R^{n+1}$. In particular $N(x) = \abs{x}^2 = x\overline{x}$, so that every non-zero vector $x$ is invertible with $x^{-1}=\abs{x}^{-1}\overline{x}$. It follows that the set $\Gamma_n$ of all products of vectors is a multiplicative group, called the \emph{Clifford group}. Another characterization of $\Gamma_n$ is the set of all invertible elements $a\in C_n$ so that $a V^{n+1}a'^{-1}=V^{n+1}$. In fact, the map $x\mapsto axa'^{-1}$ is an Euclidean isometry of $V^{n+1}$.   

\subsubsection{Hyperbolic isometries and the group $\SV_n(\R)$}

To describe the connection between Clifford algebras and hyperbolic isometries, let $\HH^{n+2}$ denote the upper half-space model for the hyperbolic $(n+2)$-space, which is  
\begin{equation*}
    \HH^{n+2}=\{x=x_0+x_1i_1+\dots+x_ni_n+x_{n+1}i_{n+1}\in V^{n+2}\mid x_{n+1}>0\}
\end{equation*}
with the hyperbolic metric $ds^2=x_{n+1}^{-2}\left(dx_0^2+\dots+dx_{n+1}^2\right)$. 

We let $\SV_n(\R)$ denote the set of all $(2\times 2)$-matrices $g=\begin{pmatrix}a&b\\c&d\end{pmatrix}$ with entries in $\Gamma_n\cup\{0\}$, satisfying $ab^*,cd^*\in V^{n+1}$ and $\mu(g) := ad^*-bc^*=1$ ($\mu$ is called the \emph{pseudo-determinant}). 
Then it turns out $\SV_n(\R)$ is a group, with the usual matrix multiplication, and it acts by isometries on $\HH^{n+2}$ by 
\begin{equation}\label{action1}
g\cdot x := (ax+b)(cx+d)^{-1}, 
\end{equation}
where we used the obvious embedding $C_n\subset C_{n+1}$. In fact, \eqref{action1} induces a short exact sequence
\begin{equation*}
    1 \to \{\pm 1\} \to SV_n(\R) \to \mathrm{Iso}^+(\HH^{n+2}) \to 1,
\end{equation*}
where ${\mathrm{Iso}^+(\HH^{n+2})}$ is the group of orientation-preserving isometries of $\H^{n+2}$. Using the hyperboloid model for $\HH^{n+2}$ we also have ${\mathrm{Iso}^+(\HH^{n+2})} \isom SO^+(1, n+2)$ (the identity component of $\SO(1, n+2)$) and $SV_n(\R) \isom \Spin(1, n+2)$.

\subsection{The lattice \texorpdfstring{$\SV_2(\ZZ)$}{}}\label{sec:lattice}

We now restrict the discussion to the case $n=2$. 
In other words, we consider the action of $\SV_2(\RR)$ on the upper half-space model for the hyperbolic $4$-space $\HH^4$. By convention we will identify an element $z\in\HH^4$ with its coordinates with respect to the basis $\{1,i_1,i_2,i_3\}$ of $V^4$, writing    
\begin{equation}\label{eq:coordinates}
    z=(x_0,x_1,x_2,y)\text{ with } y>0
\end{equation}
for an element $x_0+x_1i_1+x_2i_2+yi_3\in\HH^4$. 
In this case, $\SV_2(\RR)$ consists of matrices with entries in $\Gamma_2 \cup \{0\}$. It is easy to check that $\Gamma_2 \cup \{0\} = C_2$, which from now on we identify with the Hamilton quaternions $\bH$. This readily leads to
\begin{equation*}
    \SV_2(\R) = \left\{ \begin{pmatrix} a & b \\ c & d \end{pmatrix} \in M_2(\bH) \mid \begin{pmatrix} a & b \\ c & d \end{pmatrix} \begin{pmatrix} 0 & 1 \\ -1 & 0 \end{pmatrix} \begin{pmatrix} a^* & c^* \\ b^* & d^* \end{pmatrix} = \begin{pmatrix} 0 & 1 \\ -1 & 0 \end{pmatrix} \right\}.
\end{equation*}
Thus it is clear how to view $\SV_2(\RR)$ as the group of $\RR$-points of an affine algebraic group $\SV_2$ over $\Z$. We consider the group of $\ZZ$-points $\SV_2(\ZZ)\subset \SV_2(\RR)$. Concretely, $\SV_2(\ZZ)$ is the subgroup of $\SV_2(\R)$ consisting of matrices whose entries belong to the Lipschitz integral quaternions 
\begin{equation*}
\bH(\Z) = \Z+\Z i+\Z j+\Z k.
\end{equation*}
Then $\SV_2(\ZZ)$ is a non-uniform lattice in $\SV_2(\R)$, and we let $X$ denote the quotient space\footnote{This definition is consistent with the one given in \cref{subsection:intro_results}, since the image of $\SV_2(\Z)$ on $\Isom^+(\HH^4)$ is the group $\Gamma$ described there \cite[Proposition 3.1]{Pit05}.} 
$$
X=\SV_2(\Z)\bs\HH^4
$$
with the probability measure $dz$ coming from the Riemannian measure $d\vol(z)=dx_0dx_1dx_2dy/y^4$ on $\HH^4$.
By \cite[Proposition 1]{Kri88}, the set
\begin{equation*}
    \mathcal{F} = \left\{ z \in \HH^4 \mid -\frac{1}{2}\le x_0\le\frac{1}{2},\quad 0 \leq x_1, x_2 \leq \frac{1}{2},\quad\abs{z} \geq 1 \right\}
\end{equation*}
is a fundamental domain for the action of $\SV_2(\ZZ)$ (or more precisely of $\PSV_2(\ZZ)$) on $\HH^4$. In particular, we have $dz=\frac{d\vol}{\vol(\cF)}$. 
For $T\geq 1$ consider the regions
\begin{equation*}
    \mathcal{S}_T = \left\{ z \in \cF \mid y \ge T \right\}
\end{equation*}
and
\begin{equation*}
    \widetilde{\mathcal{S}}_T = \left\{ z \in \HH^4 \mid -\frac{1}{2} \leq x_0,x_1,x_2 \leq \frac{1}{2},\quad y \geq T \right\}.
\end{equation*}
While $\cS_T$ (for $T\to \infty$) parametrizes the cusp, it will be more convenient to work with $\widetilde{\cS}_T$ as it is more symmetric. We can easily do so since
\begin{equation}\label{eq:S_contained_in_fin_many_F}
    \widetilde{\mathcal{S}}_T = \mathcal{S}_T \cup \begin{pmatrix} i & 0\\ 0 & i' \end{pmatrix} \mathcal{S}_T \cup \begin{pmatrix} j & 0\\ 0 & j' \end{pmatrix} \mathcal{S}_T \cup \begin{pmatrix} k & 0\\ 0 & k' \end{pmatrix} \mathcal{S}_T,
\end{equation}
where the four components above are pairwise disjoint away from their boundaries, which have Riemannian volume zero. 

\subsection {Maass forms and their Fourier expansion}\label{sec:cusp-forms}

With our usual choice of coordinates \eqref{eq:coordinates} the Laplace--Beltrami operator is given by  
\begin{equation*}
    \Delta=y^2\left(\partial_{x_0}^2 + \partial_{x_1}^2 + \partial_{x_2}^2 + \partial_{y}^2\right) - 2 y\partial_y.
\end{equation*}

In our context, an automorphic form $\phi$ is an $\SV_2(\Z)$-invariant eigenfunction of $\Delta$ satisfying the moderate growth condition $\abs{\phi(z)}\ll y^M$ as $y \to \infty$, for some constant $M$. Writing $\lambda=(3/2)^2+r^2$ for the $\Delta$-eigenvalue, $\phi$ has a Fourier expansion (see e.g.\ \cite{Maa49})
\begin{equation}\label{eq:fourier-expansion}
    \phi(z) = \phi_0(y) + \sum_{0 \not= \beta \in V^3(\ZZ)} A(\beta) y^{3/2} K_{ir}\left(2 \pi \sqrt{N(\beta)} y\right) e\left(\mathrm{Re}(\beta z)\right),
\end{equation}
where $V^3(\Z) := \Z + i\Z + j\Z$, $K_{ir}$ denotes the $K$-Bessel function, and
\begin{equation*}
    \phi_0(y) = 
    \begin{cases}
        A_1 y^{3/2+ir} + A_2  y^{3/2-ir} & \text{ if } r \not= 0,\\
        A_1  y^{3/2} + A_2  y^{3/2}\log{y} & \text{ if } r = 0.
    \end{cases}
\end{equation*}

\begin{convention}\label{convention}
    We will always view the function $A(\cdot)$ as a function on all of $V^3$, rather than just on $V^3(\ZZ)$, by setting $A(\beta)=0$ whenever $\beta\notin V^3(\ZZ)$. 
\end{convention}
The form $\phi$ is called cuspidal, or a cusp form, if $A_1=A_2=0$, so that $\phi_0$ is identically zero.

\subsection {Hecke--Maass cusp forms}\label{subsection:Hecke_operators}

To define Hecke operators we will use the group of similitudes,  
$$
{\rm{GSV}}^+_2(\RR)=\left\{g=\begin{pmatrix}a&b\\ c&d \end{pmatrix}\in M_2(\bH) \mid ad^*-bc^*=\mu(g) \in \R^+, ab^*,dc^*\in V^3\right\}.  
$$
Like $\SV_2(\R)$, it acts on $\HH^4$ by $\eqref{action1}$.
Suppose that $g\in {\rm{GSV}}^+_2(\RR)$ belongs to the commensurator ${\rm{Comm}}_{{\rm{GSV}}^+_2(\RR)}(\SV_2(\ZZ))$ of $\SV_2(\Z)$ in ${\rm{GSV}}^+_2(\RR)$, i.e.\ $\SV_2(\Z)\bs (\SV_2(\Z)g\SV_2(\Z))$ is finite. 
The corresponding Hecke operator $T$ on $L^2(X)$ is defined by 
\begin{equation}\label{hecke-op1}
    Tf(x)=\sum_{s\in S}f(sx),     
\end{equation}
where $S\subset \SV_2(\Z) g \SV_2(\Z)$ is a fixed set of representatives for $\SV_2(\Z)\bs \SV_2(\Z) g\SV_2(\Z)$. We will occasionally use the notation $T\sim g$.

To each prime number $p\ne2$ we shall associate two Hecke operators $T_1(p)$ and $T_2(p)$; they will play a central role in the proof of \cref{thm:main}. To define them we follow the presentation of Pitale from \cite{Pit05} and choose, for a fixed odd prime $p$, an element $\hat{\alpha}$ ($=\hat{\alpha}(p)$) of $\bH(\ZZ)$ such that $N(\hat{\alpha}) = p$ and $p$ does not divide $\hat{\alpha}^n$ for all $n\ge1$. 
Let $T_1(p)$ and $T_2(p)$ denote the operators corresponding to the following diagonal matrices: 
\begin{equation}\label{eq:3Tp}
    T_1(p)\sim \begin{pmatrix}1&0\\0&p\end{pmatrix}\qquad \text{ and } \qquad T_2(p)\sim \begin{pmatrix}\hat{\alpha}&0\\0&\hat{\alpha}'p \end{pmatrix}.
\end{equation}
(Note that $T_1(p)$ is what Pitale calls $T_p$ and $T_2(p)$ is what Pitale calls $T_{p^2}$.)

Let $\cH_p$ denote the $\CC$-algebra of operators generated by $T_1(p)$, $T_2(p)$, and the identity operator $I$. We have $T_1(p)\cdot T_2(p)=T_2(p) \cdot T_1(p)$, so that $\cH_p$ is commutative. Also, Hecke operators corresponding to different primes commute, so that the $\C$-algebra generated by all the $\cH_p$ is commutative. We denote this algebra by $\cH$ and call it the \emph{Hecke algebra}. The Hecke operators, being defined by isometries, commute with $\Delta$. A cusp form $\phi$ is called a {\it Hecke--Maass cusp form} if it is a joint eigenfunction of $\cH$ (thus a joint eigenfunction of $\Delta$ and $\cH$). 

For what comes next it will be convenient to define a third operator $T_3(p)$ by the relation
\begin{equation}\label{eq:Hecke_relation_no_normalization}
    T_1(p)^2 - (p+1)T_2(p) - T_3(p) = \left(1+p+p^2+p^3\right)I. 
\end{equation}

\begin{remark}
    In fact $T_3(p)$ corresponds to a double coset, namely
    \begin{equation}\label{eq:4T_p}
        T_3(p)\sim \begin{pmatrix}1&0\\0&p^2\end{pmatrix},
    \end{equation}
    and this forms our main motivation for using $T_1(p),T_2(p),T_3(p)$. However, we will not use \eqref{eq:4T_p}. 
\end{remark}


\section{The comparison inequality}

\subsection{Statement of result}

As noted in the introduction, to prove \cref{thm:non-escape_of_mass}, we closely follow the strategy of \cite{Sou10}.
The main technical result of the paper is the following theorem, which is the analog of \cite[Theorem 3]{Sou10} in our context.   

\begin{theorem}\label{thm:main}
    There exist absolute constants $C, R > 0$ such that the following holds. Let $A(\beta)$ be Fourier coefficients of a Hecke--Maass cusp form on $X$ as in \eqref{eq:fourier-expansion}. Then for any $1\leq y\leq x$, 
    \begin{equation*}
        \sum_{\substack{0 \not= \beta \in V^3(\Z) \\ N(\beta) \leq x/y}} \abs{A(\beta)}^2 \leq C\frac{(1+\log{y})^R}{y^{1/8}} \sum_{\substack{0 \not= \beta \in V^3(\Z) \\ N(\beta) \leq x}} \abs{A(\beta)}^2.
    \end{equation*}
\end{theorem}

\subsection{\texorpdfstring{\cref{thm:main} implies \cref{thm:non-escape_of_mass}}{}}

The proof of \cref{thm:main} will occupy most of the remainder of the paper; before going any further, let us show how to derive \cref{thm:non-escape_of_mass} from it. 
We shall make use of the regions $\mathcal{S}_T$ and $\widetilde{\mathcal{S}}_T$ defined in \cref{sec:lattice}. 

\begin{proposition}\label{prop:bound_for_mass_in_cusp}
    There exist $R,C>0$ such that for every $T\ge1$ and Hecke--Maass cusp form $\phi\in L^2(X)$ with $\norm{\phi}_2=1$,  
    \begin{equation*}
        \int_{\mathcal{S}_T} \abs{\phi(z)}^2 dz \le C\frac{(1+\log{T})^R}{T^{1/4}}.
    \end{equation*}
    
\end{proposition}

\begin{proof}
    Let $\phi$ be as above and for each $T\ge1$ set $\mathcal{I}_T(\phi) := \int_{\mathcal{S}_T} \abs{\phi(z)}^2 dz$. Using \eqref{eq:S_contained_in_fin_many_F} and the Fourier expansion \eqref{eq:fourier-expansion} of $\phi$ we obtain
    \begin{align*}
        4 \vol(\cF) \cdot \mathcal{I}_T(\phi) & = \int_{\widetilde{\mathcal{S}}_T} \abs{\phi(z)}^2 d\vol(z) \\
        &= \int_{T}^\infty \int_{-\frac{1}{2}}^{\frac{1}{2}} \int_{-\frac{1}{2}}^{\frac{1}{2}} \int_{-\frac{1}{2}}^{\frac{1}{2}} \abs{\phi(z)}^2 dx_0 dx_1 dx_2 \frac{dy}{y^4} \\
        & = \int_T^\infty \sum_{0 \ne \beta \in V^3(\ZZ)} \abs{A(\beta)}^2  \abs{K_{ir}\left(2 \pi \sqrt{N(\beta)} y\right)}^2 \frac{dy}{y} \\
        & = \sum_{0 \ne \beta \in V^3(\ZZ)} \abs{A(\beta)}^2 \int_{T \sqrt{N(\beta)}}^\infty  \abs{K_{ir}\left(2 \pi y\right)}^2 \frac{dy}{y} \\
        &= \int_{1}^\infty  \Bigg(\sum_{\substack{0 \ne \beta \in V^3(\ZZ) \\ N(\beta) \leq y^2/T^2}} \abs{A(\beta)}^2\Bigg)\cdot \abs{K_{ir}\left(2 \pi y\right)}^2 \frac{dy}{y}.
    \end{align*}

    Applying \cref{thm:main} to the sum over $\beta$ shows that there are absolute constants $R,C>0$ such that
    \begin{equation}\label{eq:comparison-eq}
        \mathcal{I}_T(\phi) \leq C \frac{(1+\log{T})^R}{T^{1/4}} \mathcal{I}_1(\phi).
    \end{equation} 
    On the other hand $\mathcal{I}_1(\phi)\leq \norm{\phi}_2^2 = 1$. Thus the result follows from \eqref{eq:comparison-eq}.
    
\end{proof}

\begin{proof}[Proof of \cref{thm:non-escape_of_mass}.]
    Let $\phi_j \in L^2(X)$ be a sequence of joint eigenfunctions of $\cH$ and $\Delta$ with unit norm and Laplace eigenvalues $\lambda_j \to \infty$. Suppose that the probability measures $\mu_j = \abs{\phi_j(x)}^2 dx$ converge weak-* to a measure $\mu$. 
    Since the $\phi_j$ are in the discrete spectrum of $\Delta$ and the residual spectrum is finite\footnote{In fact, from the Fourier expansion of the Eisenstein series given in \cite[Eqs.\ (12) and (14)]{Kri88} it follows that the only residual form for $\SV_2(\Z) \bs \H^4$ is the constant function.}, by passing to a subsequence we may assume that the $\phi_j$ are all Hecke--Maass cusp forms. Thus the result follows from \cref{prop:bound_for_mass_in_cusp} by approximating the $\mu$-measure of $X$ by the measures of the bounded sets $X\setminus \cS_T$ as $T\to\infty$.   
    
\end{proof}


\section{Two lemmas on integral quaternions}\label{sec:basic-arithmetic}

We collect some facts on integral quaternions that will be used in the sequel. 
For $\gamma \in \bH(\Z)$ and $m \in \Z$, write $m \mid \gamma$ if $m$ divides each of the four coordinates of $\gamma$. If $q$ is a prime, $v_q(\gamma)$ denotes the largest integer such that $q^{v_q(\gamma)} \mid \gamma$. 
Now fix a prime number $p>2$. 
Recall that there are $8(p+1)$ integral quaternions (that is, elements of $\bH(\Z)$) of norm $p$, and fix a set of representatives
\begin{equation*}
\{\alpha_1,\dots,\alpha_{p+1}\}\subset \bH(\Z)
\end{equation*}
for the orbits of the left action of the group of unit quaternions $\{\pm 1, \pm i, \pm j, \pm k\}$ on them.  
\begin{lemma}\label{lemma:v_p_of_conjugate}
Let $\beta\in V^3(\ZZ)$. 
\begin{itemize}
\item For all $\alpha\in \bH(\ZZ)$ with $N(\alpha)=p$ we have 
    \begin{equation*}
        v_p(\beta) \leq v_p(\alpha' \beta \overline{\alpha}) \leq v_p(\beta) + 2.
    \end{equation*}
\item For all $\alpha$ as in the previous item and any odd prime $q \ne p$ we have 
    \begin{equation*}
        v_q(\beta) = v_q(\alpha' \beta \overline{\alpha})
    \end{equation*}
\item There are at most two distinct $i\in \{1, 2, \dots, p+1\}$ such that
    \begin{equation*}
        v_p(\alpha_i' \beta \overline{\alpha_i}) \ne v_p(\beta).
    \end{equation*}
\end{itemize}
\end{lemma}

\begin{proof}
    This follows directly from \cite[Proposition 5.8 (2)]{Pit05}.
    
\end{proof}

As a corollary we get the following useful fact.
\begin{lemma}\label{lemma:quaternion_multiplicity}
Let $\delta \in V^3(\ZZ)$. If there exists a subset $S\subset \{\alpha\in \bH(\ZZ)\mid N(\alpha)=p\}$ with $\abs{S} > 16$ such that $p^2\mid \alpha^* \delta \alpha$ for each $\alpha\in S$, then $p^2 \mid \delta$.
\end{lemma}
\begin{proof}
Since each orbit of the unit quaternions as above has size $8$, it follows that $\overline{S} := \{ \overline{\alpha} \mid \alpha \in S\}$ must intersect at least three distinct such orbits. Then there are three distinct $i \in \{1, 2, \dots, p+1\}$ such that $p^2 \mid (\overline{\alpha_i})^* \delta \overline{\alpha_i} = \alpha_i'\delta \overline{\alpha_i}$. Thus the result follows from the last item of \cref{lemma:v_p_of_conjugate}.

\end{proof}


\section{The arithmetic structure of Fourier coefficients of Hecke--Maass cusp forms} \label{section:arithmetic-properties}

From now on we will work with a single Hecke--Maass cusp form $\phi$, and $p$ will denote an odd prime number. Let $\lambda_\ell(p)$ denote the eigenvalue of $\phi$ under the normalized operator $p^{-e_\ell} T_\ell(p)$, i.e.
\begin{equation*}
    p^{-e_\ell}T_\ell(p) \phi = \lambda_\ell(p)  \phi,
\end{equation*}
where $e_1 = 3/2$, $e_2 = 2$, and $e_3 = 3$. For future reference we state the Hecke relation \eqref{eq:Hecke_relation_no_normalization} in terms of these eigenvalues, where it becomes 
\begin{equation}\label{eq:Hecke_relation}
    \lambda_1(p)^2 - \left(1+ \frac{1}{p}\right) \lambda_2(p) - \lambda_3(p) = 1+\frac{1}{p}+\frac{1}{p^2} + \frac{1}{p^3}.
\end{equation}

The proof of \cref{thm:main} crucially (and in some sense solely) relies on the following consequence of \eqref{eq:Hecke_relation}, giving an arithmetic structure on the Fourier coefficients of a Hecke--Maass cusp form. 
We view it as the analog of the Hecke multiplicativity property that was used in \cite{Sou10}. 
To state it, retain the notation of \cref{sec:basic-arithmetic} and let $\{\alpha_1,\dots,\alpha_{p+1}\}$ be a fixed choice of representatives for the orbits of the left action of the unit quaternions on the set of integral quaternions of norm $p$.  

\begin{lemma}
Let $\phi$ be a Hecke--Maass cusp form with Fourier expansion
\begin{equation*}
    \phi(z)=\sum_{0 \not= \beta \in V^3(\ZZ)} A(\beta) y^{3/2} K_{ir}\left(2 \pi \sqrt{N(\beta)} y\right) e\left(\mathrm{Re}(\beta z)\right).
\end{equation*}

\begin{lemmalist}
    \item \label{lemma:T1_relation}
    For each $\beta\in V^3(\Z)$ and odd prime $p$,
    $$\lambda_1(p) A(\beta) = A(p\beta) + A(\beta/p) + \frac{1}{\sqrt{p}} \sum_{i=1}^{p+1} A\left(\frac{\alpha'_i \beta \overline{\alpha_i}}{p}\right).$$

    \item \label{lemma:T2_relation}
    For each $\beta\in V^3(\Z)$ and odd prime $p$,
    $$\lambda_2(p) A(\beta) = \frac{1}{\sqrt{p}} \sum_{i=1}^{p+1} \left[ A(\alpha_i' \beta \overline{\alpha_i}) +  A\left(\frac{\alpha'_i \beta \overline{\alpha_i}}{p^2}\right)\right] + \mathcal{E}(\beta, p) A(\beta),$$ 
    where
    \begin{equation*}
        \mathcal{E}(\beta, p) := \frac{1}{p^2} \times \begin{cases}
        p^2-1 & \quad \text{ if } p \mid \beta, \\
        -1 & \quad \text{ if } p \mid N(\beta) \text{ and } p \nmid \beta,\\
        p-1 & \quad \text{ if } \left(\frac{-N(\beta)}{p}\right) = 1, \\
        -p-1 & \quad \text{ if } \left(\frac{-N(\beta)}{p}\right) = -1. 
        \end{cases}
    \end{equation*}

    \item \label{lemma:T3_relation}

    For each $\beta\in V^3(\Z)$ and odd prime $p$,
    \begin{align*}
        \lambda_3(p) A(\beta) &= A(p^2\beta) + A(\beta) \cdot \left(\1_p(\beta) - \frac{p+1}{p}\mathcal{E}(\beta, p) - \frac{p^2+p+1}{p^3}\right) + A\left(\frac{\beta}{p^2}\right) \\
        & \quad + \frac{1}{\sqrt{p}} \sum_{i=1}^{p+1} \left[ A\left(\alpha'_i \beta \overline{\alpha_i}\right) \cdot \left( \1_p\left(\alpha'_i \beta \overline{\alpha_i}\right) - \frac{1}{p}\right) + A\left(\frac{\alpha'_i \beta \overline{\alpha_i}}{p^2}\right)\cdot \left( \1_p(\beta) - \frac{1}{p} \right) \right] \\
        & \quad + \frac{1}{p} \sum_{j =1}^{p+1}\sum_{i =1}^{p+1} A\left(\frac{\alpha'_j\alpha'_i \beta \overline{\alpha_i} \: \overline{\alpha_j}}{p^2}\right) \cdot \1_p(\alpha'_i \beta \overline{\alpha_i}),
    \end{align*}
    where $\1_p(\gamma)$ denotes the indicator function of $p\mid \gamma$.
\end{lemmalist}
\end{lemma}

\begin{proof}
    The first two items are Propositions 5.8 and 5.11 of \cite{Pit05}. The last item follows from a direct computation combining the first two items and the Hecke relation \eqref{eq:Hecke_relation}.
 
\end{proof}


\section{Bounding Fourier coefficients along multiples of an integer}\label{sec:sums-along}

For a positive integer $d$ and $z>0$ we set
\begin{equation*}
    S_d(z) := \sum_{\substack{0 \not= \beta \in V^3(\Z) \\ N(\beta) \leq z \\ d \mid \beta}} \abs{A(\beta)}^2 = \sum_{\substack{0 \not= \delta \in V^3(\Z) \\ N(\delta) \leq z/d^2}} \abs{A(d \delta)}^2 \qquad\text{ and } \qquad S(z) := S_1(z).
\end{equation*}

The following proposition is the main result of this section. 

\begin{proposition}\label{prop:main-sums-along-multiples}
    There exists an absolute constant $A>0$ such that for any prime number $p>2$, integers $k, c \geq 1$ with $p \nmid c$, and real number $z>0$ we have
    \begin{equation*}
        S_{cp^k}(z) \leq A^k \Bigg( \sum_{\substack{a, b\geq 0 \\ a+2b \leq k}} \abs{\lambda_1(p)}^{2a} \abs{\lambda_2(p)}^{2b} \Bigg) S_c\left(\frac{z}{p^{2k}}\right).
    \end{equation*}
\end{proposition}

Before turning to the proof of \cref{prop:main-sums-along-multiples}, we record the following immediate consequence.

\begin{corollary}\label{cor:general_divisor}
    There exists an absolute constant $A>0$ such that for every odd integer $d \geq 1$ and real number $z>0$ we have
    \begin{equation*}
        S_d(z) \leq \left(\prod_{p\mid d} A^{v_p(d)}
        \Bigg( \sum_{\substack{a, b\geq 0 \\ a+2b \leq v_p(d)}} \abs{\lambda_1(p)}^{2a} \abs{\lambda_2(p)}^{2b} \Bigg)\right) S\left(\frac{z}{d^2}\right).
    \end{equation*}
\end{corollary}

\begin{proof}
    Follows directly from repeated application of \cref{prop:main-sums-along-multiples} for each prime factor of $d$. 
    
\end{proof}

The main idea for the proof of \cref{prop:main-sums-along-multiples} is to use the Hecke relations to reduce the size of the common divisor. This is done via the recursive inequalities in the next result. To state it, for a prime number $p>2$ and an integer $\ell \geq 0$ denote
\begin{equation*}
    R^{p, \ell}_{d}(z) := \frac{1}{p}\sum_{\substack{0\not=\beta \in V^3(\Z) \\ N(\beta) \leq z \\ d \mid \beta}} \abs{\sum_{i=1}^{p+1} A\left(\frac{\alpha_i' \beta \overline{\alpha_i}}{p^\ell}\right) }^2.
\end{equation*}

\begin{lemma}\label{lem:S_to_R-general}
For any prime number $p>2$, integers $d\ge1,\ell \ge 0$, and real number $z>0$ we have:
\begin{lemmalist}
    \item \label{lemma:S_to_R}
    
    $$
        S_{dp}(z) \ll \abs{\lambda_1(p)}^2  S_{d}\left(\frac{z}{p^2}\right) + S_{\frac{d}{(d, p)}}\left(\frac{z}{p^4}\right) + R^{p, 1}_{d}\left(\frac{z}{p^2}\right).
    $$
    \item \label{lemma:R_to_R}
    
    $$
        R^{p, \ell}_{dp^\ell}(z) \ll \left(\abs{\lambda_2(p)}^2 + 1\right) S_{d}\left(\frac{z}{p^{2\ell}}\right) + R^{p, 2}_{d}\left(\frac{z}{p^{2\ell}}\right).
    $$
    \item \label{lemma:R_to_S}
    For any integers $k\geq 0$ and $c \geq 1$ with $p\nmid c$, 
    \begin{equation*}
        R^{p, \ell}_{cp^k}(z) \ll  S_{cp^{\max(0, k-\ell)}}\left(\frac{z}{p^{2(\ell-1)}}\right).
    \end{equation*}
\end{lemmalist}
\end{lemma}

We emphasize that the implied constants in \cref{lem:S_to_R-general} are absolute. 

\begin{proof}
    Using \cref{lemma:T1_relation} and Cauchy--Schwarz, we get
    \begin{align*}
         S_{pd}(z) & = \sum_{\substack{0 \not= \beta \in V^3(\Z) \\ N(\beta) \leq z/p^2 \\ d \mid \beta}} \abs{A(p\beta)}^2 = \sum_{\substack{0 \not= \beta \in V^3(\Z) \\ N(\beta) \leq z/p^2 \\ d \mid \beta}} \left|\lambda_1(p)A(\beta) - A(\beta/p) - \frac{1}{\sqrt{p}} \sum_{i=1}^{p+1} A\left(\frac{\alpha_i'\beta\overline{\alpha_i}}{p}\right)\right|^2 \\
         & \leq 3 \sum_{\substack{0 \not= \beta \in V^3(\Z) \\ N(\beta) \leq z/p^2 \\ d \mid \beta}} \left(\abs{\lambda_1(p)}^2 \abs{A(\beta)}^2 + \abs{A(\beta/p)}^2 + \frac{1}{p}\abs{\sum_{i=1}^{p+1} A\left(\frac{\alpha_i'\beta\overline{\alpha_i}}{p}\right)}^2 \right) \\
         & \ll\abs{\lambda_1(p)}^2 \cdot S_{d}\left(\frac{z}{p^2}\right) + S_{\frac{d}{(d, p)}}\left(\frac{z}{p^4}\right)  + \frac{1}{p} \sum_{\substack{0 \not= \beta \in V^3(\Z) \\ N(\beta) \leq z/p^2 \\ d \mid \beta}}\abs{\sum_{i=1}^{p+1} A\left(\frac{\alpha_i'\beta\overline{\alpha_i}}{p}\right)}^2,
    \end{align*}
    where for the middle term we changed variables $\beta\mapsto p\beta$ and used the fact that $d \mid p\beta$ if and only if $\frac{d}{(d, p)} \mid \beta$. This gives \cref{lemma:S_to_R}.
    
    To prove the second item, note first that $R^{p, \ell}_{dp^\ell}(z) = R^{p, 0}_{d}\left(z/p^{2\ell}\right)$. Using \cref{lemma:T2_relation} we obtain
    \begin{align*}
        R^{p, 0}_{d}\left(\frac{z}{p^{2\ell}}\right) &= \sum_{\substack{0\not=\beta \in V^3(\Z) \\ N(\beta) \leq z/p^{2\ell} \\ d \mid \beta}} \abs{\frac{1}{\sqrt{p}} \sum_{i=1}^{p+1} A\left(\alpha_i' \beta \overline{\alpha_i}\right) }^2 \\
        &= \sum_{\substack{0\not=\beta \in V^3(\Z) \\ N(\beta) \leq z/p^{2\ell} \\ d \mid \beta}} \abs{\lambda_2(p)A(\beta) - \underbrace{\mathcal{E}(\beta, p)}_{\ll 1}A(\beta) -\frac{1}{\sqrt{p}} \sum_{i=1}^{p+1} A\left(\frac{\alpha_i' \beta \overline{\alpha_i}}{p^2}\right) }^2 \\
        & \ll \sum_{\substack{0\not=\beta \in V^3(\Z) \\ N(\beta) \leq z/p^{2\ell} \\ d \mid \beta}} \left( \abs{\lambda_2(p)}^2 \abs{A(\beta)}^2 + \abs{A(\beta)}^2 + \frac{1}{p} \abs{\sum_{i=1}^{p+1} A\left(\frac{\alpha_i' \beta \overline{\alpha_i}}{p^2}\right) }^2 \right),
    \end{align*}
    and this immediately gives \cref{lemma:R_to_R}. 

    Finally, we prove the last item. For $0 \not=\beta\in V^3(\Z)$ denote
    \begin{equation*}
        I(\beta) :=  \left\{ 1\leq i\leq p+1 : v_p(\alpha_i'\beta \overline{\alpha_i}) \geq v_p(\beta) + 1 \right\}. 
    \end{equation*}
    Then \cref{lemma:v_p_of_conjugate} gives $\abs{I(\beta)} \leq 2$. Using this and Cauchy--Schwarz we obtain
    $$
    \abs{\sum_{i=1}^{p+1} A\left(\frac{\alpha_i'\beta\overline{\alpha_i}}{p^\ell}\right)}^2\ll \sum_{i\in I(\beta)} \abs{A\left(\frac{\alpha_i'\beta\overline{\alpha_i}}{p^\ell}\right)}^2 + p\sum_{i \not\in I(\beta)} \abs{A\left(\frac{\alpha_i'\beta\overline{\alpha_i}}{p^\ell}\right)}^2,
    $$
    so that   
    \begin{equation}\label{eq:first-bound}
        R^{p, \ell}_{cp^k}(z)
         \ll \sum_{\substack{0 \not= \beta \in V^3(\Z) \\ N(\beta) \leq z \\ cp^k \mid \beta}} \left[\frac{1}{p}\sum_{i \in I(\beta)} \abs{A\left(\frac{\alpha_i'\beta\overline{\alpha_i}}{p^\ell}\right)}^2 + \sum_{i \not\in I(\beta)} \abs{A\left(\frac{\alpha_i'\beta\overline{\alpha_i}}{p^\ell}\right)}^2\right].
    \end{equation}
  
    Observe that for each $i\in\{1,\dots,p+1\}$ and $\beta$ as above we have $cp^{\max(0, k-\ell)} \mid \frac{\alpha_i'\beta\overline{\alpha_i}}{p^\ell}$, since $(c, p)=1$, and $N\left(\frac{\alpha_i'\beta\overline{\alpha_i}}{p^\ell}\right) = N(\beta)p^{-2(\ell-1)} \leq z p^{-2(\ell-1)}$.  
    Thus the RHS of \eqref{eq:first-bound} is equal to 
    \begin{equation}\label{eq:delta-form}
    \sum_{\substack{0 \not= \delta \in V^3(\Z) \\ N(\delta) \leq z/p^{2(\ell-1)} \\ cp^{\max(0, k-\ell)} \mid \delta}} h(\delta) \cdot \abs{A(\delta)}^2,
    \end{equation}
    for certain quantities $h(\delta) \geq 0$, and to finish the proof it suffices to show that $h(\delta) \ll 1$. 
    For each $\delta$ appearing in \eqref{eq:delta-form} we see from \eqref{eq:first-bound} that 
    \begin{equation}\label{eq:2term-bound}
        h(\delta)\le \frac{m_1(\delta)}{p}+m_2(\delta),
    \end{equation}
    where 
    \begin{align*}
        m_1(\delta) &= \# \left\{ (\beta, i) :  \beta \in V^3(\Z), i\in I(\beta) \text{ satisfy } \frac{\alpha_i'\beta\overline{\alpha_i}}{p^\ell} = \delta \right\} \\
        &= \#\Big\{i\in\{1, \dots, p+1\} : v_p(\delta)+1 \geq v_p(\alpha_i^*\delta\alpha_i) \geq 2-\ell \Big\}
    \end{align*}
    and
    \begin{align*}
        m_2(\delta) &=  \# \left\{ (\beta, i) :  \beta \in V^3(\Z), i\not\in I(\beta) \text{ satisfy } \frac{\alpha_i'\beta\overline{\alpha_i}}{p^\ell} = \delta \right\} \\
        &\leq \#\Big\{i\in\{1,\dots,p+1\} : v_p(\alpha_i^*\delta\alpha_i)> v_p(\delta)+1\Big\}.
    \end{align*}
    
    For $m_1(\delta)$ we use the obvious bound $m_1(\delta)\le p+1$. Applying \cref{lemma:quaternion_multiplicity} to $\widetilde{\delta} := p^{-v_p(\delta)} \delta \in V^3(\Z)$, since $v_p(\widetilde{\delta}) = 0$, it follows that $m_2(\delta)\le 16$. Therefore \eqref{eq:2term-bound} implies $h(\delta) \ll 1$, which completes the proof of \cref{lemma:R_to_S}.
    
\end{proof}

\begin{proof}[Proof of \cref{prop:main-sums-along-multiples}.]
    We will prove the stronger statement that if $B$ is a sufficiently large absolute constant, then for any integers $k\geq 0$ and $c\geq 1$ with $p\nmid c$, and real number $z>0$, 
    \begin{equation}\label{eq:stronger-result-1}
        \max\left(S_{cp^k}(z), R^{p, 2}_{c p^k}(z) \right)\leq B^{k+1}\Bigg( \sum_{\substack{a, b\geq 0 \\ a+2b \leq k}} \abs{\lambda_1(p)}^{2a} \abs{\lambda_2(p)}^{2b} \Bigg) S_c\left(\frac{z}{p^{2k}}\right)
    \end{equation}
    and
    \begin{equation}\label{eq:stronger-result-2}
        R^{p, 1}_{c p^k}(z) \leq B^{k+1}\Bigg( \sum_{\substack{a, b\geq 0 \\ a+2b \leq k+1}} \abs{\lambda_1(p)}^{2a} \abs{\lambda_2(p)}^{2b} \Bigg) S_c\left(\frac{z}{p^{2k}}\right).
    \end{equation}
    
    The proof is by induction on $k$. For the base cases $k=0$ and $k=1$, first note that \cref{lemma:R_to_S} implies
    \begin{align}
        R^{p, 1}_{c}(z) &\ll S_c(z), \label{eq:R_1_base} \\
        R^{p, 2}_{c}(z)&\ll S_c\left(z/p^2\right), \label{eq:R_2_base} \\
        R^{p, 2}_{cp}(z)&\ll S_c\left(z/p^2\right). \nonumber
    \end{align}
    From \eqref{eq:R_2_base} and \cref{lemma:R_to_R} we also obtain
    \begin{equation*}
        R^{p, 1}_{cp}(z) \ll (\abs{\lambda_2(p)}^2+1) \cdot S_c\left(z/p^2\right).
    \end{equation*}
    Furthermore, by \cref{lemma:S_to_R} and \eqref{eq:R_1_base} we have
    \begin{equation*}
        S_{cp}(z) \ll \abs{\lambda_1(p)}^2\cdot S_c\left(\frac{z}{p^2}\right) + S_c\left(\frac{z}{p^4}\right) + S_c\left(\frac{z}{p^2}\right).
    \end{equation*}
    Combining those five inequalities, we see that \eqref{eq:stronger-result-1} and \eqref{eq:stronger-result-2} hold for $k=0$ and $k=1$, as long as $B$ is a sufficiently large absolute constant. 
    
    The induction step for $k\geq 2$ follows directly from \cref{lemma:S_to_R} and \cref{lemma:R_to_R}, as long as $B$ is large enough in terms of the implied constants in that lemma (which are absolute), concluding the proof.
    
\end{proof}


\section{Proof of \texorpdfstring{\cref{thm:main}}{}: auxiliary results}

The proof of \cref{thm:main} requires the following three lemmas, which we derive from the results of \cref{sec:sums-along}.

\begin{lemma}\label{lemma:large_eigenvalues}
    For any prime $p > 2$, real number $z>0$, and $\ell\in\{1,2\}$,
    \begin{equation*}
        \abs{\lambda_\ell(p)}^2 \cdot S\left(z\right) \ll S\left(zp^2\right).
    \end{equation*}
\end{lemma}

\begin{proof}
    The case $\ell = 1$ follows by applying  \cref{lemma:T1_relation}, Cauchy--Schwarz, and \cref{lemma:R_to_S} to get
    \begin{align*}
        \abs{\lambda_1(p)}^2 \cdot S\left(z\right) \ll S\left(zp^2\right) +S(z/p^2) +  R^{p, 1}_1(z)\ll
        S\left(zp^2\right) +S(z/p^2) +S(z)\ll S(zp^2).
    \end{align*}
    The case $\ell = 2$ follows by applying \cref{lemma:T2_relation}, Cauchy--Schwarz, and \cref{lemma:R_to_S} to get 
    \begin{align*}
        \abs{\lambda_2(p)}^2 \cdot S\left(z\right) 
        \ll R^{p, 0}_1(z) + R^{p, 2}_1(z) + S\left(z\right)\ll S(zp^2)+S(z/p^2)+S(z)\ll S(zp^2).
    \end{align*}

\end{proof}

For the rest of this section, consider a real parameter $P\geq 1$, and an arbitrary subset 
\begin{equation*}
    \P \subset \left\{p \in \left[\frac{P}{2}, P\right] : p \text{ is an odd (rational) prime}\right \}.
\end{equation*}
Following \cite{Sou10}, to bound the sums $S(z)$, we will distinguish between $A(\beta)$ with $\beta$ divisible by ``many" elements of $\P$ and those with $\beta$ divisible by ``not as many" elements of $\P$. To make this distinction precise we define the following subsets of $V^3(\Z)$, analogous to the sets $\cN_j(k)$ in \cite[Section 3]{Sou10}. For $K>0$ and a positive integer $\ell$, set
\begin{equation}\label{def:multiplicity}
    \M_\ell(K) := \left\{0 \not= \beta\in V^3(\Z) : \text{there are}\leq K \text{ primes } p \in \P \text{ such that } p^\ell\mid \beta\right\}.
\end{equation}

\begin{lemma}\label{lemma:sum_of_conjs_mult_p_bound}
	For any real numbers $z>0$ and $K\geq 1$,
	\begin{equation}\label{eq:sum_of_conj1}
		\sum_{\substack{\beta \in \M_1(K) \\ N(\beta) \leq z }}\sum_{\substack{p\in \P \\ p\nmid \beta}} \left|  \frac{1}{\sqrt{p}}\sum_{i=1}^{p+1} A\left(\frac{\alpha_i' \beta \overline{\alpha_i}}{p}\right)\right|^2 \ll K \cdot  S\left(z\right)
	\end{equation}
	and 
	\begin{equation}\label{eq:sum_of_conj2}
		\sum_{\substack{\beta \in \M_1(K) \\ N(\beta) \leq z }}\sum_{\substack{p\in \P \\ p\nmid \beta}} \left|  \frac{1}{\sqrt{p}}\sum_{i=1}^{p+1} A\left(\frac{\alpha_i' \beta \overline{\alpha_i}}{p^2}\right)\right|^2 \ll |\P| \cdot  S\left(\frac{z}{(P/2)^2}\right).
	\end{equation}
\end{lemma}

\begin{proof}
	By \cref{lemma:v_p_of_conjugate}, the inner sum in both \eqref{eq:sum_of_conj1} and \eqref{eq:sum_of_conj2} has at most two non-zero terms. Thus by Cauchy--Schwarz \eqref{eq:sum_of_conj1} and \eqref{eq:sum_of_conj2} are 
\begin{align}\label{eq:cauchy-inner}
    &\leq  \sum_{\substack{\beta \in \M_1(K) \\ N(\beta) \leq z }}\sum_{\substack{p\in \P \\ p\nmid \beta}} \frac{2}{p} \sum_{i=1}^{p+1} \left|A\left(\frac{\alpha_i' \beta \overline{\alpha_i}}{p^\ell}\right)\right|^2,
\end{align}
where $\ell=1,2$ respectively. \cref{lemma:v_p_of_conjugate} also yields that each of the elements $\frac{\alpha_i'\beta\overline{\alpha_i}}{p^\ell}$ appearing in \eqref{eq:sum_of_conj1} and \eqref{eq:sum_of_conj2} belongs to $\M_1(K+1)$, so the quantity in \eqref{eq:cauchy-inner} is 
\begin{equation}\label{eq:cauchy-outer}
    \ll \frac{1}{P} \sum_{\substack{\delta \in \M_1(K+1) \\ N(\delta) \leq z/(P/2)^{2(\ell-1)}}} m(\delta, \ell) \cdot |A(\delta)|^2,
\end{equation}
where 
\begin{equation*}
    m(\delta, \ell) = \# \left\{ (\beta, p, i) : \beta \in V^3(\Z), p\in \P, 1\leq i\leq p+1 \text{ satisfy } p \nmid \beta \text{ and }\frac{\alpha_i' \beta \overline{\alpha_i}}{p^\ell} = \delta \right\}. 
\end{equation*}

We now treat the cases $\ell=1$ and $\ell=2$ separately, starting with the former one. 
Fix $0\not= \delta \in V^3(\Z)$. Note that there are at most $p+1\leq P+1$ pairs $(\beta, i)$ with $\frac{\alpha_i' \beta \overline{\alpha_i}}{p} = \delta$, since $p$ and $i$ determine $\beta$. For each choice of $p \in \P$, either there are in fact at most $16$ pairs $(\beta, i)$ with $\frac{\alpha_i' \beta \overline{\alpha_i}}{p} = \delta$, or \cref{lemma:quaternion_multiplicity} implies that $p \mid \delta$, since $p^2 \mid \alpha_i^* (p\delta) \alpha_i$ for each of these pairs. Thus if $\delta\in\M_1(L)$ then 
	\begin{equation*}
    	m(\delta,1) \leq (P+1)L + 16(|\P|-L).
	\end{equation*}
	Since $K\geq 1$ this leads to
	\begin{align*}
    		\frac{1}{P} \sum_{\substack{\delta \in \M_1(K+1) \\ N(\delta) \leq z}} m(\delta,1) \cdot |A(\delta)|^2 \ll \frac{KP+|\P|}{P} \sum_{\substack{\delta \in \M_1(K+1) \\ N(\delta) \leq z }}|A(\delta)|^2 \ll K  \cdot S\left(z\right),
	\end{align*}
	concluding the proof of \eqref{eq:sum_of_conj1}.
        For the case $\ell=2$ we use the obvious bound
        \begin{equation*}
		m(\delta,2) \leq (P+1) \cdot |\P|,
        \end{equation*}
        which (since $P\geq 1$) yields
        \begin{equation*}
		\frac{1}{P} \sum_{\substack{\delta \in \M_1(K+1) \\ N(\delta) \leq z/(P/2)^2}} m(\delta,2) \cdot |A(\delta)|^2 \ll \frac{P|\P|}{P} \sum_{\substack{\delta \in \M_1(K+1) \\ N(\delta) \leq z/(P/2)^2 }}|A(\delta)|^2 \ll |\P|  \cdot S\left(\frac{z}{(P/2)^2}\right).
        \end{equation*}
    
\end{proof}

Denote
\begin{equation}\label{eq:before-lemma7}
    S_\ell^\flat(z,K) := \sum_{\substack{\beta \notin \M_\ell(K) \\ N(\beta) \leq z }}|A(\beta)|^2.
\end{equation}

\begin{lemma}\label{lemma:flat_sum_bounds}
    There exists an absolute constant $B$ such that the following holds. For any integers $\ell, K\geq 1$ and real number $z>0$ we have
    \begin{equation*}
        S_\ell^\flat(z,K) \leq \left(\frac{B^\ell \cdot |\P| \cdot \sup_{p\in\P}\mathcal{L}_\ell(p)}{K+1}\right)^{K+1} S\left(\frac{z}{(P/2)^{2\ell (K+1)}}\right),
    \end{equation*}
    where
    \begin{equation*}
        \mathcal{L}_\ell(p) := \sum_{\substack{a, b\geq 0 \\ a+2b \leq \ell}} \abs{\lambda_1(p)}^{2a} \abs{\lambda_2(p)}^{2b}.
    \end{equation*}
\end{lemma}

\begin{proof}
    Since every element $\beta$ which appears in $S_\ell^\flat(z, K)$ is a multiple of at least $K+1$ distinct $\ell$-th powers of primes in $\P$,
    \begin{align*}
        S_\ell^\flat(z, K)& \leq \sum_{\substack{d = (p_1 p_2 \cdots p_{K+1})^\ell \\ p_1 < p_2 < \dots < p_{K+1} \\
            p_i \in \P}}  \sum_{\substack{0 \not= \beta \in V^3(\Z) \\ N(\beta) \leq z \\ d \mid \beta}} |A(\beta)|^2 =  \sum_{\substack{d = (p_1 p_2 \cdots p_{K+1})^\ell \\ p_1 < p_2 < \dots < p_{K+1} \\
            p_i \in \P}}  S_{d}(z). 
    \end{align*}
    Observe that $d \geq (P/2)^{\ell (K+1)}$. Applying \cref{cor:general_divisor} gives, for some absolute constant $A>0$,
    \begin{align*}
        S_\ell^\flat(z, K) &\leq \binom{|\P|}{K+1} A^{\ell(K+1)}\left(\sup_{p\in\P}\mathcal{L}_\ell(p)\right)^{K+1} S\left(\frac{z}{(P/2)^{2\ell (K+1)}}\right) \\
        &\leq \left(\frac{e \cdot |\P| \cdot  A^\ell\cdot \sup_{p\in\P}\mathcal{L}_\ell(p)}{K+1}\right)^{K+1} S\left(\frac{z}{(P/2)^{2\ell (K+1)}}\right).
    \end{align*}

\end{proof}


\section{Concluding the proof of \texorpdfstring{\cref{thm:main}}{}}

We are finally ready to prove \cref{thm:main}. Recall that the partial sums $S$ were defined using an eigenfunction $\phi$. The content of \cref{thm:main} is that the functions $y\mapsto S(x/y)/S(x)$ decay uniformly in $\phi$. The following general result provides a uniform rate of decay for any collection of compactly supported functions $f:[1,\infty)\to [0, 1]$ satisfying a single recursive inequality of a certain kind. We will prove \cref{thm:main} by showing that our functions $f(y) = \frac{S(x/y)}{S(x)}$ satisfy such a recursive inequality.
\begin{lemma}\label{lemma:recursion_asymptotics}
    Let $f:[1, \infty) \to [0, 1]$ be a compactly supported function with $f(1) = 1$. Consider real numbers $\Delta > 0, \eps \in (0, 1)$, and $A \geq 10$, integers $M, N \geq 0$, and functions $a_m, b_n : [1, \infty) \to \R$, for each $1\leq m\leq M$ and $1\leq n \leq N$, satisfying 
    \begin{align*}
        1 \geq a_m(y) \geq \eps \quad \quad \text{ and } \quad \quad  b_n(y) \geq \eps (1+\log{y})^{\eps}.
    \end{align*}
    Assume that for every $y\geq A$ we have
    \begin{equation*}
        f(y) \leq A \left[ \frac{(\log{y})^A}{y^\Delta} + f\left(y^{1+\eps}\right) + \sum_{m=1}^M y^{-\Delta a_m(y)} \cdot f\left(y^{1- a_m(y)}\right) + \sum_{n=1}^N e^{-\eps b_n(y)} y^{\Delta b_n(y)}\cdot f\left(y^{1+ b_n(y)}\right) \right].
    \end{equation*}
    Then there exist parameters $C$ and $R$ which depend only on $\{A, M, N, \Delta,\eps\}$ such that
    \begin{equation*}
        f(y) \leq C \cdot \frac{(1+\log{y})^R}{y^\Delta}
    \end{equation*}
    for every $y \geq 1$. In particular, if $A, M, N, \Delta,\eps$ are all absolute constants, then so are $C$ and $R$.
\end{lemma}

\begin{remark}
    \cref{lemma:recursion_asymptotics} holds under the weaker assumption that $b_m(y)\geq h(y)$ for an arbitrary function satisfying $h(y)\to\infty$ as $y\to\infty$, in which case the parameters $C$ and $R$ would also depend on $h$. We state the result as above for concreteness.
\end{remark}

\begin{proof}
    Define $g :[1, \infty) \to \R_{\geq 0}$ by $g(y) = y^\Delta \cdot f(y)$. Then for every $y \geq A$,
    \begin{equation}\label{eq:g_recursive_ineq}
        g(y) \leq A \left[(\log{y})^A \cdot g(1) + \frac{g\left(y^{1+\eps}\right)}{y^{\eps\Delta}} + \sum_{m=1}^M g\left(y^{1- a_m(y)}\right) + \sum_{n=1}^N e^{-\eps b_n(y)} \cdot g\left(y^{1+ b_n(y)}\right) \right].
    \end{equation}

    Since $f$ is compactly supported and bounded, for any real number $r\geq 0$ there exists a real number $z_r\geq 1$ such that
    \begin{equation*}
        \frac{g(z_r)}{(1+\log{z_r})^r} \geq \frac{1}{2}\cdot \sup_{y\geq 1} \frac{g(y)}{(1+\log{y})^r}.
    \end{equation*}
    Thus for any $y \geq 1$,
    \begin{equation}\label{eq:g_maximality_ineq}
        g(y) \leq 2\left(\frac{1+\log{y}}{1+\log{z_r}}\right)^r g(z_r).
    \end{equation}
    We have $g(z_r) > 0$, since $g(1)=1$.

    Let $R$ denote the smallest integer satisfying $R\geq A \geq 10$,
    \begin{equation*}
        2A(\log{A})^{A-R} \leq \frac{1}{4}, \qquad \text{ and } \qquad 2AM \left(1 - \frac{\eps}{2}\right)^R \leq \frac{1}{4},
    \end{equation*}
    so that $R \ll_{A, M, \eps} 1$.

    \subsection*{Case 1:} $z_R \geq \max\left(A, \left[8A(1+\eps)^R\right]^{\frac{1}{\eps\Delta}}\right)$.

    In this case we may apply \eqref{eq:g_recursive_ineq} followed by \eqref{eq:g_maximality_ineq} to obtain
    \begin{align*}
        g(z_R) &\leq 2A \left[\frac{(\log{z_R})^A}{(1+\log{z_R})^{R}} + \left(1 + \frac{ \eps\log{z_R}}{1+\log{z_R}}\right)^R \frac{1}{(z_R)^{\eps\Delta}} + \sum_{m=1}^M \left(1 - a_m(z_R) \cdot \frac{ \log{z_R}}{1+\log{z_R}}\right)^R\right. \\
        & \qquad \qquad \qquad \qquad \qquad \quad + \left. \sum_{n=1}^N e^{-\eps b_n(z_R)} \cdot \left(1 + b_n(z_R) \cdot \frac{ \log{z_R}}{1+\log{z_R}}\right)^R \right] g(z_R) \\
        & \leq 2A \left[(\log{A})^{A-R} + \frac{\left(1+\eps\right)^R}{8A\left(1+\eps\right)^R} + M \left(1 - \frac{\eps}{2}\right)^R + \sum_{n=1}^N e^{-\eps b_n(z_R)} \cdot \Big(1 + b_n(z_R) \Big)^R \right] g(z_R).
    \end{align*}
    Since $g(z_R) > 0$, this implies
    \begin{equation*}
        \frac{1}{4} \leq 2A\sum_{n=1}^N e^{-\eps b_n(z_R)} \cdot \Big(1 + b_n(z_R) \Big)^R.
    \end{equation*}
    
    There is a parameter $B>0$ depending only on $R$ and $\eps$, and therefore only on $\{A, M, \eps\}$, such that $e^{-\eps x} (1+x)^R \leq B e^{-\frac{\eps}{2}x}$ for every $x\geq 0$. Therefore, 
    \begin{equation*}
        \frac{1}{4} \leq 2AB \sum_{n=1}^N e^{-\frac{\eps}{2}b_n(z_R)} \leq 2ABN e^{-\frac{\eps^2}{2}(1+\log{z_R})^{\eps}}.
    \end{equation*}

    This shows that $z_R \ll_{A, M, N, \eps} 1$. Since $f$ is bounded by $1$, condition \eqref{eq:g_maximality_ineq} for $r=R$ shows that for any $y \geq 1$ we have
    \begin{equation}\label{eq:asymp_ine_last_step}
        f(y) \leq 2 \left(1+\log{y}\right)^R \left(\frac{z_R}{y}\right)^\Delta f(z_R) \leq 2 \frac{\left(1+\log{y}\right)^R}{y^\Delta} (z_R)^\Delta \ll_{A, M, N, \Delta, \eps} \frac{\left(1+\log{y}\right)^R}{y^\Delta},
    \end{equation}
    as desired.

    \subsection*{Case 2:} $z_R < \max\left(A, \left[8A(1+\eps)^R\right]^{\frac{1}{\eps\Delta}}\right)$.

    In this case, since $R \ll_{A, M, \eps} 1$ we have $z_R \ll_{A, M, \Delta, \eps} 1$, so arguing as in \eqref{eq:asymp_ine_last_step} we conclude that
    \begin{equation*}
        f(y) \ll_{A, M, \Delta, \eps} \frac{\left(1+\log{y}\right)^R}{y^\Delta}
    \end{equation*}
    and the result follows.
    
\end{proof}

\begin{proposition}\label{prop:main_technical_inequality}
    There exists an absolute constant $D>0$ such that the following holds. Let $x\geq 1$, $y\geq 10^{10^{10}}$, and denote $\nu := \frac{1}{8}$ and $P := y^\nu$. There exist positive integer parameters $K_\ell$ satisfying
    \begin{equation*}
        K_\ell \geq \log{P}
    \end{equation*}
    for each $\ell \in \{1, 2, 3, 4\}$ such that
    \begin{align*}
        S\left(\frac{x}{y}\right) \leq D &\left[ \frac{(\log{P})^5}{P}\cdot S(x) + S\left(\frac{x}{y(P/2)^2}\right) + \sum_{\ell=1}^3 P^{-2\ell\nu} \cdot S\left(\frac{xP^{2\ell}}{y}\right) \right. \\
         &\qquad \qquad \qquad \left. + \sum_{\ell=1}^4 e^{-(K_\ell+1)} \left(\frac{P}{2}\right)^{2\ell \nu (K_\ell+1)} \cdot S\left(\frac{x}{y(P/2)^{2\ell(K_\ell+1)}}\right) \right].
    \end{align*}
\end{proposition}

Let us see how to conclude once we have \cref{prop:main_technical_inequality}.

\begin{proof}[Proof of \cref{thm:main}.]
    We will work with a single value of $x \geq 1$ in the proof. If $S(x) = 0$ the result is trivial, so we assume otherwise.

    The theorem then follows from applying \cref{lemma:recursion_asymptotics} to the compactly supported function $f(y) := \frac{S(x/y)}{S(x)}$, with parameters $\Delta = \frac{1}{8}$, $A$ a sufficiently large absolute constant, $\eps>0$ a sufficiently small absolute constant, $M=3$, and $N=4$. Indeed, \cref{prop:main_technical_inequality} shows that $f$ satisfies the required conditions for \cref{lemma:recursion_asymptotics} with (say) $\eps = \frac{1}{100}$ and $A = \max\left(D, 10^{10^{10}}\right)$.
    
\end{proof}

\begin{proof}[Proof of \cref{prop:main_technical_inequality}.]
    If $x < y$ then $S(x/y)=0$ and there is nothing to prove, so we assume $x\geq y\geq 10^{10^{10}}$. 
    Let $\cQ$ denote the set of primes contained in the interval $\left[\frac{P}{2}, P\right]$.

    If there exists $p\in\cQ$ such that $\abs{\lambda_\ell(p)}^2\ge y$ for some $\ell\in\{1,2,3\}$, then 
    the Hecke relation \eqref{eq:Hecke_relation} implies that either $\abs{\lambda_1(p)}^2$ or $\abs{\lambda_2(p)}^2$ is greater than $y^{1/3}$, so \cref{lemma:large_eigenvalues} gives
    \begin{equation*}
        S\left(\frac{x}{y}\right) \ll y^{-1/3}\cdot S\left(\frac{xP^2}{y}\right),
    \end{equation*}
    and the desired inequality follows (with plenty of room). Thus from now on we assume that  
    \begin{equation}\label{eq:range-of-eigenvalues}
        \abs{\lambda_\ell(p)}^2<y
    \end{equation}
    for all $p\in\cQ$ and $\ell \in\{1,2,3\}$. For $i\geq 1$ and $\ell \in \{1, 2, 3\}$, set
    \begin{equation*}
        \cQ^\ell_i := \left\{p\in \cQ : \frac{2^{i-1}}{100}<|\lambda_\ell(p)|^2 \leq \frac{2^i}{100} \right\}, \qquad \qquad     \cQ^\ell_0 := \left\{p\in \cQ : |\lambda_\ell(p)|^2 \leq \frac{1}{100} \right\},
    \end{equation*}
    and
    $$
    \P_{ijk}:= \cQ^1_i\cap \cQ^2_j \cap \cQ^3_k.
    $$
    By \eqref{eq:range-of-eigenvalues} we have
    \begin{equation*}
    \cQ = \bigsqcup_{\substack{0\leq i, j, k \leq J }} \P_{ijk}
    \end{equation*}
    for (say) $J := 2 \log y$. 
    From $y\geq 10^{10^{10}}$ it follows that
    \begin{equation*}
        |\cQ| \geq \frac{P}{2 \log{P}},
    \end{equation*}
    so there exists a tuple $(i,j,k)$ such that 
    \begin{equation*}
        |\P_{ijk}| \geq \frac{|\cQ|}{(J+1)^3} \geq \frac{P}{2\log{P}\cdot (3\log{y})^3} \geq\frac{P}{10^2\cdot(\log{P})^4}.
    \end{equation*}
    We fix such a tuple and observe that $(i,j,k)\ne(0,0,0)$, due to the Hecke relation \eqref{eq:Hecke_relation}.
    Denote $\P := \P_{ijk}$. 
    The argument splits into three cases.


\subsection*{Case 1} $i > 0$.

Let $L := \frac{2^i}{100}$, so that $L\geq \frac{1}{50}$ and
\begin{equation*}
    \frac{L}{2} < |\lambda_1(p)|^2 \leq L
\end{equation*}
for every $p\in\P$. Let $K_1$ be any integer satisfying  
\begin{equation}\label{eq:K_basic_ineqs}
    1 \leq K_1 \leq \frac{|\P|}{2},
\end{equation} 
and denote 
\begin{equation*}
    S^\sharp(z) := \sum_{\substack{\beta \in \M_1(K_1) \\ N(\beta) \leq z }}|A(\beta)|^2 \quad \quad \text{ and } \quad \quad S^\flat(z) := \sum_{\substack{\beta \not\in \M_1(K_1) \\ N(\beta) \leq z }}|A(\beta)|^2.
\end{equation*}
(Note that $S^\flat(z)$ is $S_1^\flat\left(z, K_1\right)$ from \eqref{eq:before-lemma7}, but we set this more general notation for the cases ahead.) 
Then $S=S^\sharp+S^\flat$ and we bound $S(x/y)$ by bounding $S^\sharp(x/y)$ and $S^\flat(x/y)$ separately, always assuming \eqref{eq:K_basic_ineqs}.  

To bound $S^\sharp(x/y)$, we will amplify using $\lambda_1(p)$. Consider the expression
\begin{align*}
    \mathcal{A} := \sum_{\substack{\beta \in \M_1(K_1) \\ N(\beta) \leq x/y }}|A(\beta)|^2 \cdot \left( \sum_{\substack{p\in \P \\ p\nmid \beta}} |\lambda_1(p)|^2\right).
\end{align*}
Since $|\lambda_1(p)|^2 \geq \frac{L}{2}$ for every $p\in\P$ and $K_1 \leq \frac{|\P|}{2}$,
\begin{equation}\label{eq:larger-than}
    \mathcal{A} \geq \frac{L(|\P|-K_1)}{2} S^\sharp\left(\frac{x}{y}\right) \gg L|\P| \cdot S^\sharp\left(\frac{x}{y}\right).
\end{equation}
On the other hand, by the Hecke relation for $\lambda_1(p) A(\beta)$,
\begin{align*}
    \mathcal{A} &= \sum_{\substack{\beta \in \M_1(K_1) \\ N(\beta) \leq x/y }}\sum_{\substack{p\in \P \\ p\nmid \beta}} \left| A(p\beta) + \underbrace{A(\beta/p)}_{=0} + \frac{1}{\sqrt{p}}\sum_{i=1}^{p+1} A\left(\frac{\alpha_i' \beta \overline{\alpha_i}}{p}\right)  \right|^2 \\
    & \ll \sum_{\substack{\beta \in \M_1(K_1) \\ N(\beta) \leq x/y }}\sum_{\substack{p\in \P \\ p\nmid \beta}} | A(p\beta)|^2 +  \sum_{\substack{\beta \in \M_1(K_1) \\ N(\beta) \leq x/y }}\sum_{\substack{p\in \P \\ p\nmid \beta}} \frac{1}{p} \left|  \sum_{i=1}^{p+1} A\left(\frac{\alpha_i' \beta \overline{\alpha_i}}{p}\right)\right|^2.
\end{align*}
The first double sum is $\leq(K_1+1) \cdot S\left(\frac{xP^2}{y}\right)$. Indeed, if $\beta \in \M_1(K_1)$ then $\delta = p\beta \in \M_1(K_1+1)$, therefore for each such $\delta$ there are at most $K_1+1$ choices for $p \in \P$. 
The second double sum is $\ll K_1 \cdot S(x/y)$, due to \cref{lemma:sum_of_conjs_mult_p_bound}. Combining this with \eqref{eq:larger-than} we conclude that
\begin{align}\label{eq:sharp_sum_bound}
    S^\sharp\left(\frac{x}{y}\right) \ll \frac{K_1}{L|\P|}\cdot S\left(\frac{xP^2}{y}\right).
\end{align}

To bound $S^\flat(x/y)$ we first apply \cref{lemma:flat_sum_bounds} and get
\begin{equation}\label{eq:flat_sum_bound}
    S^\flat\left(\frac{x}{y}\right) = S_1^\flat\left(\frac{x}{y}, K_1\right) \leq \left(\frac{B L |\P|}{K_1+1}\right)^{K_1+1} S\left(\frac{x}{y(P/2)^{2(K_1+1)}}\right)
\end{equation}
for some absolute constant $B\geq 1$. If $L \geq \frac{P^{2\nu}}{10^{10}B}$ then \cref{lemma:large_eigenvalues} trivially implies the desired result, so we assume otherwise. Then we can take $K_1 := \left\lceil e \cdot \frac{BL|\P|}{(P/2)^{2\nu}} \right\rceil - 1 \ge\log P$ and \eqref{eq:K_basic_ineqs} is satisfied. Applying \eqref{eq:sharp_sum_bound} and \eqref{eq:flat_sum_bound} we conclude that
\begin{align*}
    S\left(\frac{x}{y}\right) & \ll {P^{-2\nu}} \cdot S\left(\frac{xP^2}{y}\right) + \left(\frac{B L |\P|}{K_1+1}\right)^{K_1+1} S\left(\frac{x}{y(P/2)^{2(K_1+1)}}\right) \\
    & \leq {P^{-2\nu}} \cdot S\left(\frac{xP^2}{y}\right) + e^{-(K_1+1)}\left(\frac{P}{2}\right)^{2\nu(K_1+1)} S\left(\frac{x}{y(P/2)^{2(K_1+1)}}\right). 
\end{align*}

This completes the proof of Case 1. 


\subsection*{Case 2} $i=0$ and $j > 0$.

In this case we have
\begin{equation*}
    |\lambda_1(p)|^2 \leq \frac{1}{100} \qquad \text{ and } \qquad \frac{L}{2} < |\lambda_2(p)|^2 \leq L
\end{equation*}
for every $p\in\P$, where $L:= \frac{2^{j}}{100} \geq \frac{1}{50}$. Our argument depends on the size of $L$. 

\subsection*{Subcase 2.1} $L > 10^{10}$.

Consider integers 
\begin{equation}\label{eq:K_12_P_bound}
    1 \leq K_1 \leq \frac{|\P|}{2} \qquad \text{ and } \qquad K_2 \geq 1
\end{equation}
to be chosen later. Let $\M(\Vec{K}) := \M_1(K_1) \cap \M_2(K_2)$ and denote
\begin{equation*}
    S^\sharp(z) := \sum_{\substack{\beta \in \M(\Vec{K}) \\ N(\beta) \leq z }}|A(\beta)|^2 \quad \quad \text{ and } \quad \quad S^\flat(z) := \sum_{\substack{\beta \not\in \M(\Vec{K}) \\ N(\beta) \leq z }}|A(\beta)|^2.
\end{equation*}

To bound $S^\sharp(x/y)$, we will amplify using $\lambda_2(p)$. Consider the amplified expression
\begin{align*}
    \mathcal{A} := \sum_{\substack{\beta \in \M(\Vec{K}) \\ N(\beta) \leq x/y }}|A(\beta)|^2 \cdot \left( \sum_{\substack{p\in \P \\ p\nmid \beta}} |\lambda_2(p)|^2\right).
\end{align*}
Since $|\lambda_2(p)|^2 \geq \frac{L}{2}$ for every $p\in\P$ and $K_1 \leq \frac{|\P|}{2}$,
\begin{equation*}
    \mathcal{A} \geq \frac{L(|\P|-K_1)}{2} S^\sharp\left(\frac{x}{y}\right) \geq \frac{L|\P|}{4} \cdot S^\sharp\left(\frac{x}{y}\right).
\end{equation*}

Using the Hecke relations for $\lambda_2(p)A(\beta)$ and $\lambda_1(p)A(p\beta)$, 
\begin{align*}
    \lambda_2(p)A(\beta) &= \frac{1}{\sqrt{p}}\sum_{i=1}^{p+1} \left[A\left(\alpha_i' \beta \overline{\alpha_i}\right) + A\left(\frac{\alpha_i' \beta \overline{\alpha_i}}{p^2}\right) \right] + \mathcal{E}(\beta, p) A(\beta) \\
    &= \lambda_1(p)A(p\beta) - A(p^2\beta) + \frac{1}{\sqrt{p}}\sum_{i=1}^{p+1} A\left(\frac{\alpha_i' \beta \overline{\alpha_i}}{p^2}\right) + \left(\mathcal{E}(\beta, p) - 1\right)A(\beta).
\end{align*}
Therefore, since $|\mathcal{E}(\beta, p)| \leq 1$ and $|\lambda_1(p)|^2 \leq \frac{1}{100}$,
\begin{align*}
    \mathcal{A} &= \sum_{\substack{\beta \in \M(\Vec{K}) \\ N(\beta) \leq x/y }}\sum_{\substack{p\in \P \\ p\nmid \beta}} \left| \lambda_1(p)A(p\beta) - A(p^2\beta) + \frac{1}{\sqrt{p}}\sum_{i=1}^{p+1} A\left(\frac{\alpha_i' \beta \overline{\alpha_i}}{p^2}\right) + \left(\mathcal{E}(\beta, p) - 1\right)A(\beta) \right|^2 \\
    & \leq 4 \sum_{\substack{\beta \in \M(\Vec{K}) \\ N(\beta) \leq x/y }}\sum_{\substack{p\in \P \\ p\nmid \beta}} \left[ 4 \left|A(\beta)\right|^2 + \frac{\left|A(p\beta)\right|^2}{100} + \left|A(p^2\beta)\right|^2 + \left|\frac{1}{\sqrt{p}}  \sum_{i=1}^{p+1} A\left(\frac{\alpha_i' \beta \overline{\alpha_i}}{p^2}\right)\right|^2 \right] \\
    &\leq 16 |\P| \cdot  S^\sharp\left(\frac{x}{y}\right) + \frac{K_1+1}{25} \cdot S\left(\frac{xP^2}{y}\right) + 4(K_2+1) \cdot S\left(\frac{xP^4}{y}\right) \\
    & \qquad \qquad \qquad \qquad \qquad \qquad \qquad \qquad \qquad + 4\sum_{\substack{\beta \in \M(\Vec{K}) \\ N(\beta) \leq x/y }}\sum_{\substack{p\in \P \\ p\nmid \beta}}  \left|\frac{1}{\sqrt{p}}  \sum_{i=1}^{p+1} A\left(\frac{\alpha_i' \beta \overline{\alpha_i}}{p^2}\right)\right|^2 .
\end{align*}
In the last step we used the fact that $p\beta$ is a multiple of at most $K_1+1$ distinct primes in $\P$, and $p^2\beta$ is a multiple of at most $K_2+1$ squares of distinct primes in $\P$, since $\beta \in \M(\Vec{K})$.

Applying \cref{lemma:sum_of_conjs_mult_p_bound} to the last term, since $\M(\Vec{K}) \subset \M_1(K_1)$ we see that it is
\begin{equation*}
    \ll |\P| \cdot S\left(\frac{x}{y(P/2)^2}\right).
\end{equation*}
Since $L>10^{10}$, we have $\frac{L}{4}-16 \geq \frac{L}{5}$, so we obtain
\begin{equation}\label{eq:case_2_sharp_sum_bound}
    S^\sharp\left(\frac{x}{y}\right) \ll  \frac{K_1}{L|\P|} \cdot S\left(\frac{xP^2}{y}\right) + \frac{K_2}{L|\P|} \cdot S\left(\frac{xP^4}{y}\right)  + \frac{1}{L} \cdot S\left(\frac{x}{y(P/2)^2}\right).
\end{equation}

By \cref{lemma:flat_sum_bounds}, we also have
\begin{equation}\label{eq:case_2_flat_sum_bound}
    S^\flat\left(\frac{x}{y}\right) \leq S_1^\flat\left(\frac{x}{y}, K_1\right) + S_2^\flat\left(\frac{x}{y}, K_2\right) \leq \sum_{\ell=1}^2 \left( \frac{B|\P|L}{K_\ell+1} \right)^{K_\ell+1} S\left(\frac{x}{y(P/2)^{2\ell(K_\ell+1)}}\right) 
\end{equation}
for some absolute constant $B\geq 1$.

Combining \eqref{eq:case_2_sharp_sum_bound} and \eqref{eq:case_2_flat_sum_bound}, we conclude that
\begin{align*}
    S\left(\frac{x}{y}\right) \ll \frac{1}{L} \cdot S\left(\frac{x}{y(P/2)^2}\right) + \sum_{\ell=1}^2 \frac{K_\ell}{L|\P|} \cdot S\left(\frac{xP^{2\ell}}{y}\right) + \left( \frac{B|\P|L}{K_\ell+1} \right)^{K_\ell+1} S\left(\frac{x}{y(P/2)^{2\ell(K_\ell+1)}}\right)
\end{align*}

If $L \geq \frac{P^{2\nu}}{10^{10} B}$ then the desired result follows from \cref{lemma:large_eigenvalues}. Otherwise, note that we can choose 
\begin{equation*}
    K_\ell := \left\lceil e \cdot \frac{B|\P|L}{(P/2)^{2\ell\nu}} \right\rceil - 1 \geq \frac{P^{1-2\ell\nu}}{10^{5}(\log{P})^4} \geq \log{P}
\end{equation*}
for $\ell \in \{1, 2\}$, and \eqref{eq:K_12_P_bound} is satisfied. Therefore, we obtain
\begin{align*}
    S\left(\frac{x}{y}\right) \ll S\left(\frac{x}{y(P/2)^2}\right) + \sum_{\ell=1}^2 {P^{-2\ell\nu}} \cdot S\left(\frac{xP^{2\ell}}{y}\right) + e^{-(K_\ell+1)} \left(\frac{P}{2}\right)^{2\ell \nu (K_\ell+1)} \cdot S\left(\frac{x}{y(P/2)^{2\ell(K_\ell+1)}}\right).
\end{align*}

This finishes the proof of Subcase 2.1.

\subsection*{Subcase 2.2} $L \leq 10^{10}$.

Similarly to the previous subcase, we consider integers 
\begin{equation}\label{eq:K_1234_P_bound}
    1 \leq K_1 \leq \frac{|\P|}{2}, \qquad  K_2 \geq 1, \qquad  K_3 \geq 1, \qquad \text{ and } \qquad K_4 \geq 1
\end{equation}
to be chosen later. Let $\M(\Vec{K}) := \bigcap_{\ell=1}^4 \M_\ell(K_\ell)$ and denote
\begin{equation*}
    S^\sharp(z) := \sum_{\substack{\beta \in \M(\Vec{K}) \\ N(\beta) \leq z }}|A(\beta)|^2 \quad \quad \text{ and } \quad \quad S^\flat(z) := \sum_{\substack{\beta \not\in \M(\Vec{K}) \\ N(\beta) \leq z }}|A(\beta)|^2.
\end{equation*}

To bound $S^\sharp(x/y)$, we will again amplify using $\lambda_2(p)$. Consider the amplified expression
\begin{align*}
    \mathcal{A} := \sum_{\substack{\beta \in \M(\Vec{K}) \\ N(\beta) \leq x/y }}|A(\beta)|^2 \cdot \left( \sum_{\substack{p\in \P \\ p\nmid \beta}} |\lambda_2(p)|^2\right).
\end{align*}
Since $|\lambda_2(p)|^2 > \frac{L}{2} \geq \frac{1}{100}$ for every $p\in\P$ and $K_1 \leq \frac{|\P|}{2}$,
\begin{equation*}
    \mathcal{A} \geq \frac{L(|\P|-K_1)}{2} S^\sharp\left(\frac{x}{y}\right) \gg |\P| \cdot S^\sharp\left(\frac{x}{y}\right).
\end{equation*}

Using the Hecke relations for $\lambda_2(p)A(\beta)$, followed by $\lambda_2(p)A(p^2\beta)$ and $\lambda_1(p) A(p^3\beta)$, 
\begin{align*}
    \lambda_2(p)A(\beta) &= \frac{1}{\sqrt{p}}\sum_{i=1}^{p+1} \left[A\left(\alpha_i' \beta \overline{\alpha_i}\right) + A\left(\frac{\alpha_i' \beta \overline{\alpha_i}}{p^2}\right) \right] + \mathcal{E}(\beta, p) A(\beta) \\
    &= \lambda_2(p)A(p^2\beta) + \frac{1}{\sqrt{p}}\sum_{i=1}^{p+1} \left[ A\left(\frac{\alpha_i' \beta \overline{\alpha_i}}{p^2}\right) - A\left(p^2 \alpha_i' \beta \overline{\alpha_i}\right) \right] \\
    &\qquad \qquad \qquad + \mathcal{E}(\beta, p)A(\beta) - \mathcal{E}(p^2\beta, p) A(p^2 \beta)\\
    & = \lambda_2(p)A(p^2\beta) + \frac{1}{\sqrt{p}}\sum_{i=1}^{p+1}  A\left(\frac{\alpha_i' \beta \overline{\alpha_i}}{p^2}\right)  -\lambda_1(p)A(p^3\beta) + A\left(p^4 \beta\right) \\
    & \qquad \qquad \qquad  - A(p^2\beta) + \mathcal{E}(\beta, p)A(\beta) - \mathcal{E}(p^2\beta, p) A(p^2 \beta).
\end{align*}
We have $|\mathcal{E}(p^2\beta, p)| \leq 1$, and since $p\nmid \beta$ also $|\mathcal{E}(\beta, p)| \leq \frac{p+1}{p^2} \ll \frac{1}{p}$. Then using $|\lambda_1(p)|^2 \leq \frac{1}{100}$ and $|\lambda_2(p)|^2 \leq 10^{10} \ll 1$, we apply Cauchy--Schwarz to obtain
\begin{align*}
    \mathcal{A} &\ll \sum_{\substack{\beta \in \M(\Vec{K}) \\ N(\beta) \leq x/y }}\sum_{\substack{p\in \P \\ p\nmid \beta}} \left[ \left|A(p^4\beta)\right|^2 + \left|A(p^3\beta)\right|^2 + \left|A(p^2\beta)\right|^2 + \frac{\left|A(\beta)\right|^2}{p^2} + \left|\frac{1}{\sqrt{p}}  \sum_{i=1}^{p+1} A\left(\frac{\alpha_i' \beta \overline{\alpha_i}}{p^2}\right)\right|^2 \right] \\
    &\leq \sum_{\ell=2}^4 (K_\ell + 1) \cdot  S\left(\frac{xP^{2\ell}}{y}\right) + \frac{|\P|}{(P/2)^2} \cdot  S\left(\frac{x}{y}\right) + \sum_{\substack{\beta \in \M(\Vec{K}) \\ N(\beta) \leq x/y }}\sum_{\substack{p\in \P \\ p\nmid \beta}}  \left|\frac{1}{\sqrt{p}}  \sum_{i=1}^{p+1} A\left(\frac{\alpha_i' \beta \overline{\alpha_i}}{p^2}\right)\right|^2 .
\end{align*}
In the last step we used the fact that for $\ell \in \{2, 3, 4\}$, $p^\ell\beta$ is a multiple of at most $K_\ell+1$ distinct $\ell$-th powers of primes in $\P$, since $\beta \in \M(\Vec{K}) \subset \M_\ell(K_\ell)$. 

Applying \cref{lemma:sum_of_conjs_mult_p_bound} to the last term, since $\M(\Vec{K}) \subset \M_1(K_1)$ we see that it is
\begin{equation*}
    \ll |\P| \cdot S\left(\frac{x}{y(P/2)^2}\right).
\end{equation*}
Therefore
\begin{equation}\label{eq:subcase_2_2_sharp_sum_bound}
    S^\sharp\left(\frac{x}{y}\right) \ll \sum_{\ell=2}^4 \frac{K_\ell}{|\P|} \cdot  S\left(\frac{xP^{2\ell}}{y}\right) + \frac{1}{P^2} \cdot  S\left(\frac{x}{y}\right) + S\left(\frac{x}{y(P/2)^2}\right).
\end{equation}

By \cref{lemma:flat_sum_bounds}, we also have
\begin{equation}\label{eq:subcase_2_2_flat_sum_bound}
    S^\flat\left(\frac{x}{y}\right) \leq \sum_{\ell = 1}^4 S_\ell^\flat\left(\frac{x}{y}, K_\ell\right) \leq \sum_{\ell=1}^4 \left( \frac{B|\P|}{K_\ell+1} \right)^{K_\ell+1} S\left(\frac{x}{y(P/2)^{2\ell(K_\ell+1)}}\right) 
\end{equation}
for some absolute constant $B\geq 1$.

Combining \eqref{eq:subcase_2_2_sharp_sum_bound} and \eqref{eq:subcase_2_2_flat_sum_bound}, we conclude that
\begin{align*}
    S\left(\frac{x}{y}\right) &\ll \frac{1}{P^2} \cdot  S\left(\frac{x}{y}\right) + S\left(\frac{x}{y(P/2)^2}\right)  \\
    &\qquad \qquad + \sum_{\ell=1}^4 \frac{K_\ell}{|\P|} \cdot S\left(\frac{xP^{2\ell}}{y}\right) + \left( \frac{B|\P|}{K_\ell+1} \right)^{K_\ell+1} S\left(\frac{x}{y(P/2)^{2\ell(K_\ell+1)}}\right)
\end{align*}

If $P^{2\nu} \leq 10^{10} B$, then trivially
\begin{equation*}
    S\left(\frac{x}{y}\right) \leq S\left(\frac{xP^2}{y}\right) \ll  P^{-2\nu} \cdot S\left(\frac{xP^2}{y}\right),
\end{equation*}
and the result follows. Otherwise, note that we can choose 
\begin{equation*}
    K_\ell := \left\lceil e \cdot \frac{B|\P|}{(P/2)^{2\ell\nu}} \right\rceil - 1 \geq \frac{P^{1-2\ell\nu}}{10^{5}(\log{P})^4} \geq \log{P}
\end{equation*}
for $\ell \in \{1, 2, 3\}$, and \eqref{eq:K_1234_P_bound} is satisfied. We also choose
\begin{equation*}
    K_4 := \max\left(\lceil\log{P}\rceil, \left\lceil e \cdot \frac{B|\P|}{(P/2)} \right\rceil - 1 \right) 
\end{equation*}
so clearly $\log{P} \gg K_4 \geq \log{P}$. Since $P^8 = y$, we can bound
\begin{equation*}
    \frac{K_4}{|\P|} \cdot S\left(\frac{xP^{8}}{y}\right) \ll \frac{(\log{P})^5}{P} \cdot S\left(x\right).
\end{equation*}

Therefore,
\begin{align*}
    S\left(\frac{x}{y}\right) &\ll \frac{(\log{P})^5}{P} \cdot S\left(x\right) + S\left(\frac{x}{y(P/2)^2}\right) + \sum_{\ell=1}^3 {P^{-2\ell\nu}} \cdot S\left(\frac{xP^{2\ell}}{y}\right) \\
    &\qquad \qquad + \sum_{\ell=1}^4 e^{-(K_\ell+1)} \left(\frac{P}{2}\right)^{2\ell \nu (K_\ell+1)} \cdot S\left(\frac{x}{y(P/2)^{2\ell(K_\ell+1)}}\right).
\end{align*}

This finishes the proof of Subcase 2.2.


\subsection*{Case 3} $i=j=0$ and $k > 0$.

In this case we have
\begin{equation*}
    \max\left(|\lambda_1(p)|^2, |\lambda_2(p)|^2\right) \leq \frac{1}{100},
\end{equation*}
so by the Hecke relation \eqref{eq:Hecke_relation} we obtain
\begin{equation*}
    1 \leq |\lambda_3(p)|^2 \leq 4
\end{equation*}
for every $p\in\P$. Consider integers 
\begin{equation}\label{eq:K_P_bound_case_3}
    1 \leq K_1 \leq \frac{|\P|}{2} \qquad \text{ and } \qquad K_2 \geq 1
\end{equation}
to be chosen later. Let $\M(\Vec{K}) := \M_1(K_1) \cap \M_2(K_2)$ and denote
\begin{equation*}
    S^\sharp(z) := \sum_{\substack{\beta \in \M(\Vec{K}) \\ N(\beta) \leq z }}|A(\beta)|^2 \quad \quad \text{ and } \quad \quad S^\flat(z) := \sum_{\substack{\beta \not\in \M(\Vec{K}) \\ N(\beta) \leq z }}|A(\beta)|^2.
\end{equation*}

To bound $S^\sharp(x/y)$, we will amplify using $\lambda_3(p)$. Consider the amplified expression
\begin{align*}
    \mathcal{A} := \sum_{\substack{\beta \in \M(\Vec{K}) \\ N(\beta) \leq x/y }}|A(\beta)|^2 \cdot \left( \sum_{\substack{p\in \P \\ p\nmid \beta}} |\lambda_3(p)|^2\right).
\end{align*}
Since $|\lambda_3(p)|^2 \geq 1$ for every $p\in\P$ and $K_1 \leq \frac{|\P|}{2}$,
\begin{equation*}
    \mathcal{A} \geq (|\P|-K_1) \cdot S^\sharp\left(\frac{x}{y}\right) \gg |\P| \cdot S^\sharp\left(\frac{x}{y}\right).
\end{equation*}

For $p\nmid \beta$, observe that $|\mathcal{E}(\beta, p)| \leq \frac{p+1}{p^2}$ and we have the Hecke relation 
\begin{align*}
    \lambda_3(p) A(\beta) &= A(p^2\beta) - A(\beta) \cdot \left(\frac{p+1}{p}\mathcal{E}(\beta, p) + \frac{p^2+p+1}{p^3}\right) \\
    & \quad + \frac{1}{\sqrt{p}} \sum_{i=1}^{p+1} \left[ A\left(\alpha'_i \beta \overline{\alpha_i}\right) \cdot \left( \1_p\left(\alpha'_i \beta \overline{\alpha_i}\right) - \frac{1}{p}\right) - \frac{1}{p}\cdot A\left(\frac{\alpha'_i \beta \overline{\alpha_i}}{p^2}\right) \right] \\
    & \quad + \frac{1}{p} \sum_{j =1}^{p+1}\sum_{i =1}^{p+1} A\left(\frac{\alpha'_j\alpha'_i \beta \overline{\alpha_i} \: \overline{\alpha_j}}{p^2}\right) \cdot \1_p(\alpha'_i \beta \overline{\alpha_i})
\end{align*}

We can apply Cauchy--Schwarz to obtain
\begin{align*}
    \mathcal{A} &\ll \sum_{\substack{\beta \in \M(\Vec{K}) \\ N(\beta) \leq x/y }}\sum_{\substack{p\in \P \\ p\nmid \beta}} \left[ \left|A(p^2\beta)\right|^2 + \frac{\left|A(\beta)\right|^2}{p^2} + \frac{1}{p}\left|\sum_{i=1}^{p+1} A\left(\alpha'_i \beta \overline{\alpha_i}\right) \cdot \left( \1_p\left(\alpha'_i \beta \overline{\alpha_i}\right) - \frac{1}{p}\right)\right|^2\right. \\
    &  \qquad \qquad \qquad \qquad \left. + \frac{1}{p^2} \left|\frac{1}{\sqrt{p}}  \sum_{i=1}^{p+1} A\left(\frac{\alpha_i' \beta \overline{\alpha_i}}{p^2}\right)\right|^2 + \frac{1}{p^2} \left|\sum_{i =1}^{p+1} \1_p(\alpha'_i \beta \overline{\alpha_i}) \cdot \sum_{j =1}^{p+1} A\left(\frac{\alpha'_j\alpha'_i \beta \overline{\alpha_i} \: \overline{\alpha_j}}{p^2}\right) \right|^2 \right]\\
    &\leq (K_2 + 1) \cdot  S\left(\frac{xP^{4}}{y}\right) + \frac{|\P|}{(P/2)^2} \cdot  S\left(\frac{x}{y}\right) + \sum_{\substack{\beta \in \M(\Vec{K}) \\ N(\beta) \leq x/y }}\sum_{\substack{p\in \P \\ p\nmid \beta}} \left[ \frac{1}{p^2} \left|\frac{1}{\sqrt{p}}  \sum_{i=1}^{p+1} A\left(\frac{\alpha_i' \beta \overline{\alpha_i}}{p^2}\right)\right|^2 \right. \\
    &  \qquad \left. + \frac{1}{p}\left|\sum_{i=1}^{p+1} A\left(\alpha'_i \beta \overline{\alpha_i}\right) \cdot \left( \1_p\left(\alpha'_i \beta \overline{\alpha_i}\right) - \frac{1}{p}\right)\right|^2 + \frac{1}{p^2} \left|\sum_{i =1}^{p+1} \1_p(\alpha'_i \beta \overline{\alpha_i}) \cdot \sum_{j =1}^{p+1} A\left(\frac{\alpha'_j\alpha'_i \beta \overline{\alpha_i} \: \overline{\alpha_j}}{p^2}\right) \right|^2 \right].\\
\end{align*}
In the last step we used the fact that $p^2\beta$ is a multiple of at most $K_2+1$ distinct squares of primes in $\P$, since $\beta \in \M(\Vec{K}) \subset \M_2(K_2)$.

By \cref{lemma:sum_of_conjs_mult_p_bound}, since $\M(\Vec{K}) \subset \M_1(K_1)$ we see that
\begin{equation*}
    \sum_{\substack{\beta \in \M(\Vec{K}) \\ N(\beta) \leq x/y }}\sum_{\substack{p\in \P \\ p\nmid \beta}} \frac{1}{p^2} \left|\frac{1}{\sqrt{p}}  \sum_{i=1}^{p+1} A\left(\frac{\alpha_i' \beta \overline{\alpha_i}}{p^2}\right)\right|^2  \ll \frac{|\P|}{P^2} \cdot S\left(\frac{x}{y(P/2)^2}\right) \leq \frac{|\P|}{P^2} \cdot S\left(\frac{x}{y}\right).
\end{equation*}

Also notice that for each $\beta \in V^3(\Z)$ with $p\nmid\beta$, there are at most two distinct $i\in \{1, 2, \dots, p+1\}$ such that $p \mid \alpha'_i \beta \overline{\alpha_i}$, by \cref{lemma:v_p_of_conjugate}. Thus by Cauchy--Schwarz
\begin{align*}
    \sum_{\substack{\beta \in \M(\Vec{K}) \\ N(\beta) \leq x/y }}\sum_{\substack{p\in \P \\ p\nmid \beta}} \frac{1}{p}\left|\sum_{i=1}^{p+1} A\left(\alpha'_i \beta \overline{\alpha_i}\right) \cdot \left( \1_p\left(\alpha'_i \beta \overline{\alpha_i}\right) - \frac{1}{p}\right)\right|^2 & \ll \frac{1}{P}\sum_{\substack{\beta \in \M(\Vec{K}) \\ N(\beta) \leq x/y }}\sum_{\substack{p\in \P \\ p\nmid \beta}} \sum_{i=1}^{p+1} \left|A\left(\alpha'_i \beta \overline{\alpha_i}\right) \cdot \1_p\left(\alpha'_i \beta \overline{\alpha_i}\right) \right|^2 \\
    & \quad+ \frac{1}{P^2}\sum_{\substack{\beta \in \M(\Vec{K}) \\ N(\beta) \leq x/y }}\sum_{\substack{p\in \P \\ p\nmid \beta}} \sum_{i=1}^{p+1} \left| A\left(\alpha'_i \beta \overline{\alpha_i}\right) \right|^2.
\end{align*}
For each $\delta \in V^3(\Z)$ and $p\in \P$ there are at most $p+1 \leq P+1$ pairs $(\beta, i)$ with $\alpha'_i \beta \overline{\alpha_i} = \delta$, as $i$ uniquely determines $\beta$. Also, by \cref{lemma:v_p_of_conjugate} we observe that if $\beta\in \M_1(K_1)$ then $\delta \in \M_1(K_1+1)$. This implies
\begin{align*}
    \frac{1}{P^2}\sum_{\substack{\beta \in \M(\Vec{K}) \\ N(\beta) \leq x/y }}\sum_{\substack{p\in \P \\ p\nmid \beta}} \sum_{i=1}^{p+1} \left| A\left(\alpha'_i \beta \overline{\alpha_i}\right) \right|^2 \ll \frac{1}{P} \sum_{p\in \P} \sum_{\substack{\delta \in \M_1(K_1+1) \\ N(\delta) \leq xp^2/y }} |A(\delta)|^2 \ll \frac{|\P|}{P} \cdot S\left(\frac{xP^2}{y}\right)
\end{align*}
and
\begin{align*}
    \frac{1}{P}\sum_{\substack{\beta \in \M(\Vec{K}) \\ N(\beta) \leq x/y }}\sum_{\substack{p\in \P \\ p\nmid \beta}} \sum_{i=1}^{p+1} \left| A\left(\alpha'_i \beta \overline{\alpha_i}\right) \cdot \1_p\left(\alpha'_i \beta \overline{\alpha_i}\right) \right|^2 \ll \sum_{\substack{p\in \P}} \sum_{\substack{\delta \in \M_1(K_1+1) \\ N(\delta) \leq xp^2/y \\ p \mid \delta}} \left|A(\delta)\right|^2 \leq (K_1+1) \cdot  S\left(\frac{xP^2}{y}\right), 
\end{align*}
where we used the fact that  $\delta$ is a multiple of at most $K_1+1$ distinct primes in $\P$, since $\delta \in\M_1(K_1+1)$. In conclusion,
\begin{align*}
    \sum_{\substack{\beta \in \M(\Vec{K}) \\ N(\beta) \leq x/y }}\sum_{\substack{p\in \P \\ p\nmid \beta}} \frac{1}{p}\left|\sum_{i=1}^{p+1} A\left(\alpha'_i \beta \overline{\alpha_i}\right) \cdot \left( \1_p\left(\alpha'_i \beta \overline{\alpha_i}\right) - \frac{1}{p}\right)\right|^2 \ll  K_1 \cdot S\left(\frac{xP^2}{y}\right).
\end{align*}

Finally, we must properly bound the sum
\begin{equation*}
    \mathfrak{S} := \sum_{\substack{\beta \in \M(\Vec{K}) \\ N(\beta) \leq x/y }}\sum_{\substack{p\in \P \\ p\nmid \beta}} \frac{1}{p^2} \left|\sum_{i =1}^{p+1} \1_p(\alpha'_i \beta \overline{\alpha_i}) \cdot \sum_{j =1}^{p+1} A\left(\frac{\alpha'_j\alpha'_i \beta \overline{\alpha_i} \: \overline{\alpha_j}}{p^2}\right) \right|^2.
\end{equation*}
As in the previous arguments, \cref{lemma:v_p_of_conjugate} shows that 
\begin{equation*}
    \mathfrak{S} \leq 2 \sum_{\substack{p\in \P}} \frac{1}{p^2} \sum_{\substack{\beta \in \M(\Vec{K}) \\ N(\beta) \leq x/y \\ p\nmid \beta}} \sum_{i =1}^{p+1} \1_p(\alpha'_i \beta \overline{\alpha_i}) \cdot \left|\sum_{j =1}^{p+1} A\left(\frac{\alpha'_j \left(\frac{\alpha'_i \beta \overline{\alpha_i}}{p}\right) \overline{\alpha_j}}{p}\right) \right|^2,
\end{equation*}
and that $\delta := \frac{\alpha'_i \beta \overline{\alpha_i}}{p} \in V^3(\Z)$ satisfies $\delta \in \M_1(K_1+1) \cap \M_2(K_2) \subset \M_1(K_1+1) \cap \M_2(K_2+1) =: \M(\Vec{K}+1)$. Furthermore, $v_p(\delta) = 0$ or $1$. Fix $p\in \P$. If $v_p(\delta)=0$, then \cref{lemma:quaternion_multiplicity} gives
\begin{equation*}
    \#\left\{ (\beta, i) : 0 \not= \beta \in V^3(\Z), 1 \leq i \leq p+1 \text{ satisfy } \frac{\alpha'_i \beta \overline{\alpha_i}}{p} = \delta \right\} \leq 16,
\end{equation*}
while if $v_p(\delta) = 1$ the set above trivially has size $\leq p+1$, since $i$ uniquely determines $\beta$. Therefore, 
\begin{align*}
    \mathfrak{S} \ll \frac{1}{P}\sum_{\substack{\delta \in \M(\Vec{K}+1) \\ N(\delta) \leq x/y}} \sum_{\substack{p\in \P \\ p \nmid \delta}} \left|\frac{1}{\sqrt{p}}\sum_{j =1}^{p+1} A\left(\frac{\alpha'_j \delta \overline{\alpha_j}}{p}\right) \right|^2 + \sum_{\substack{\delta \in \M(\Vec{K}+1) \\ N(\delta) \leq x/y}} \sum_{\substack{p\in \P \\ v_p(\delta)=1}} \left|\frac{1}{\sqrt{p}}\sum_{j =1}^{p+1} A\left(\frac{\alpha'_j \delta \overline{\alpha_j}}{p}\right) \right|^2.
\end{align*}

By \cref{lemma:sum_of_conjs_mult_p_bound}, for the first term we have
\begin{equation*}
    \frac{1}{P}\sum_{\substack{\delta \in \M(\Vec{K}+1) \\ N(\delta) \leq x/y}} \sum_{\substack{p\in \P \\ p \nmid \delta}} \left|\frac{1}{\sqrt{p}}\sum_{j =1}^{p+1} A\left(\frac{\alpha'_j \delta \overline{\alpha_j}}{p}\right) \right|^2 \ll \frac{K_1}{P} \cdot S\left(\frac{x}{y}\right).
\end{equation*}

For the second term, we can write $\delta = p\gamma$ for $\gamma \in \M(\Vec{K}+1)$ with $p \nmid \gamma$. Then we can use the Hecke relation for $\lambda_1(p)A(p\gamma)$ and the constraint $|\lambda_1(p)|^2 \leq \frac{1}{100}$ to obtain
\begin{align*}
    \sum_{\substack{\delta \in \M(\Vec{K}+1) \\ N(\delta) \leq x/y}} \sum_{\substack{p\in \P \\ v_p(\delta)=1}} \left|\frac{1}{\sqrt{p}}\sum_{j =1}^{p+1} A\left(\frac{\alpha'_j \delta \overline{\alpha_j}}{p}\right) \right|^2 & \leq \sum_{\substack{p\in \P}} \sum_{\substack{\gamma \in \M(\Vec{K}+1) \\ N(\gamma) \leq \frac{x}{yp^2} \\ p\nmid \gamma}} \left|\frac{1}{\sqrt{p}}\sum_{j =1}^{p+1} A\left(\alpha'_j \gamma \overline{\alpha_j}\right) \right|^2 \\
    & \ll \sum_{\substack{p\in \P}} \sum_{\substack{\gamma \in \M(\Vec{K}+1) \\ N(\gamma) \leq \frac{x}{yp^2} \\ p\nmid \gamma}}  \left [ \left|A(p^2\gamma)\right|^2 + \left|\lambda_1(p)\right|^2\cdot \left|A(p\gamma)\right|^2 + \left|A(\gamma)\right|^2 \right] \\
    & \ll K_2 \cdot S\left(\frac{xP^2}{y}\right) + K_1 \cdot S\left(\frac{x}{y}\right) + |\P| \cdot S\left(\frac{x}{y(P/2)^2}\right).
\end{align*}
In the last step we used the fact that for $\ell \in \{1, 2\}$, $p^\ell\gamma$ is a multiple of at most $K_\ell+2$ distinct $\ell$-th powers of primes in $\P$, since $\gamma \in \M(\Vec{K}+1) \subset \M_\ell(K_\ell+1)$. 

Therefore, combining all of the bounds described above we conclude that
\begin{equation}\label{eq:case_3_sharp_sum_bound}
    S^\sharp\left(\frac{x}{y}\right) \ll \frac{K_2}{|\P|} \cdot  S\left(\frac{xP^4}{y}\right) + \frac{K_1}{|\P|} \cdot  S\left(\frac{xP^2}{y}\right) +  S\left(\frac{x}{y(P/2)^2}\right).
\end{equation}

To complement \eqref{eq:case_3_sharp_sum_bound}, we can apply \cref{lemma:flat_sum_bounds} to get
\begin{align}\label{eq:case_3_flat_sum_bound}
    S^\flat\left(\frac{x}{y}\right) \leq \sum_{\ell=1}^2 S_\ell^\flat\left(\frac{x}{y}, K_\ell\right) \leq \sum_{\ell=1}^2 \left(\frac{B |P|}{K_\ell+1}\right)^{K_\ell+1} S\left(\frac{x}{y(P/2)^{2\ell(K_\ell+1)}}\right)
\end{align}
for some absolute constant $B\geq 1$.

If $P^{2\nu} \leq 10^{10} B$, then trivially
\begin{equation*}
    S\left(\frac{x}{y}\right) \leq S\left(\frac{xP^2}{y}\right) \ll  P^{-2\nu} \cdot S\left(\frac{xP^2}{y}\right),
\end{equation*}
and the result follows. Otherwise, note that we can choose 
\begin{equation*}
    K_\ell := \left\lceil e \cdot \frac{B|\P|}{(P/2)^{2\ell\nu}} \right\rceil - 1 \geq \frac{P^{1-2\ell\nu}}{10^{5}(\log{P})^4} \geq \log{P}
\end{equation*}
for $\ell \in \{1, 2\}$, and \eqref{eq:K_P_bound_case_3} is satisfied. Thus summing \eqref{eq:case_3_sharp_sum_bound} and \eqref{eq:case_3_flat_sum_bound} gives
\begin{align*}
    S\left(\frac{x}{y}\right) \ll S\left(\frac{x}{y(P/2)^2}\right) + \sum_{\ell=1}^2 P^{-2\ell\nu} \cdot S\left(\frac{xP^{2\ell}}{y}\right) + e^{-(K_\ell+1)} \left(\frac{P}{2}\right)^{2\ell \nu (K_\ell+1)} \cdot S\left(\frac{x}{y(P/2)^{2\ell(K_\ell+1)}}\right).
\end{align*}

This finishes the proof of Case 3.

\end{proof}


\nocite{*}  
\bibliographystyle{abbrv}
\bibliography{references}

\end{document}

%% file: mydefs.tex
\usepackage{mathrsfs}
\usepackage{amssymb}
\usepackage{dsfont}
\usepackage{verbatim}
\usepackage{url}
\usepackage{mathtools}
\usepackage{etoolbox}
\usepackage{leftidx}

\def\Z{\mathbb {Z}}
\def\Q{\mathbb {Q}}
\def\R{\mathbb {R}}
\def\C{\mathbb {C}}

\def\isom{\simeq}

\def\bs{\backslash}

\DeclareMathOperator{\Isom}{Isom}

\DeclareMathOperator{\vol}{vol}

\newcommand\abs[1]{\left| {#1} \right|}
\newcommand\norm[1]{\left\Vert {#1} \right\Vert}

\newcommand{\xdashrightarrow}[2][]{\ext@arrow 0359\rightarrowfill@@{#1}{#2}}

\DeclareRobustCommand
  \rddots{\mathinner{\mkern1mu\raise\p@
    \vbox{\kern7\p@\hbox{.}}\mkern2mu
    \raise4\p@\hbox{.}\mkern2mu\raise7\p@\hbox{.}\mkern1mu}}

\newcommand\SL{\mathrm{SL}}
\newcommand\SV{\mathrm{SV}}
\newcommand\PSV{\mathrm{PSV}}

\newcommand\Spin{\mathrm{Spin}}

\newcommand\SO{\mathrm{SO}}

\newcommand\CC{\mathbb{C}}

\newcommand\HH{\mathbb{H}}

\newcommand\RR{\mathbb{R}}

\newcommand\ZZ{\mathbb{Z}}

\newcommand\bH{\mathbf{H}}

\DeclareMathAlphabet{\mathcal}{OMS}{cmsy}{m}{n}

\newcommand\cF{\mathcal{F}}

\newcommand\cH{\mathcal{H}}

\newcommand\cN{\mathcal{N}}

\newcommand\cQ{\mathcal{Q}}

\newcommand\cS{\mathcal{S}}

\newcommand\diff{\mathop{}\!\mathrm{d}}

\newcommand\dcross{\diff^{\kern-2pt\raisebox{1pt}{$\times$}}\kern-8pt}

\newcommand\blfootnote{\xdef\@thefnmark{}\@footnotetext}

\DeclareFontFamily{U}{wncy}{}
\DeclareFontShape{U}{wncy}{m}{n}{<->wncyr10}{}
\DeclareSymbolFont{mcy}{U}{wncy}{m}{n}
\DeclareMathSymbol{\Sha}{\mathord}{mcy}{"58}